\documentclass[12pt]{amsart}

\usepackage[T1]{fontenc}
\usepackage[utf8]{inputenc}

\usepackage{amsmath,amstext,amssymb,amsthm} 
\usepackage{enumitem,geometry,float,pgfplots}
\usepackage{tikz,standalone}
\usepackage{mathtools}
\usepackage{MnSymbol}
\usepackage{stmaryrd}
\usepackage{relsize}
\usepackage{array,multirow,makecell}
\usepackage{verbatim}
\usepackage{bbm}
\usepackage{accents}
\usepackage{fancyhdr}
\usepackage{graphicx}
\usepackage{subcaption}
\usepackage[ddmmyyyy,hhmmss]{datetime}
\usepackage[hidelinks]{hyperref} 
\usepackage{pdfpages}
\usepackage{lipsum}

\def\lhs{\textnormal{LHS}}
\def\rhs{\textnormal{RHS}}

\def\inverseFunction#1{ #1^{ -1 } }
\def\myRestriction#1#2{#1 \restriction #2} 
\def\topologicalClosure#1{ \textnormal{Cl} \, #1 }
\def\topologicalInterior#1{ \textnormal{Int} \, #1 }
\def\NN{\mathbb{N}} 
\def\ZZ{\mathbb{Z}} 
\def\RR{\mathbb{R}} 
\def\CC{\mathbb{C}} 
\def\lfrac#1#2{#1 / #2} 
\def\diff#1{d #1} 
\def\complexAdjoint#1{ \overline{#1} }
\def\complexInfty{\infty_{\CC}}
\def\supp{\textnormal{supp} \, }

\def\realLogarithm{\textnormal{ln}\,}
\def\multivaluedLogarithm{\textnormal{log}\,}
\def\principalLogarithm{\textnormal{Log}\,}
\def\principalLogarithmB{\mathcal{L}\,}

\def\principalArgument{\textnormal{Arg}\,}

\def\certainArgumentB{\textnormal{arg}\,}
\def\complexArctan{\textnormal{arctan}\,} 

\def\weaklyConvergesTo{\overset{d}{\to}} 

\def\part{\mathcal{P}} 

\def\monotone{\textnormal{M}}

\def\mStandardGauss{\mu^{\monotone}}
\def\mStandardGaussDensity{\rho^{\monotone}}

\def\part{\mathcal{P}} 

\def\domainTrapezoid{\Theta}
\def\domainOpenTrapezoid{ \topologicalInterior{\Theta} }
\def\domainMaximalTrapezoid{\Xi}
\def\domainMaximalOpenTrapezoid{ \topologicalInterior{\Xi} }
\def\domainExteriorOfTheDisk{E}
\def\domainStripZero{ \Delta_{0} }

\def\domainOpenUpperShell{\Delta}
\def\domainLowerStrip{ \topologicalInterior{D} }
\def\domainRangeOfTZero{L}
\def\domainTrapezoidUnderH{ \domainClosedLowerStrip \setminus \domainRangeOfTZero }
\def\domainClosedLowerStrip{D}

\def\theFunctionMmu{M_{\mu}}
\def\theFunctionGmu{G_{\mu}}
\def\theExtendedGmu{ \widehat{G}_{\mu} }
\def\theFunctionFmu{F_{\mu}}
\def\theExtendedFmu{ \widehat{F}_{\mu} }
\def\theRealValuedT{ T_{0} }
\def\theInverseOfRealT{ \inverseFunction{\theRealValuedT} }
\def\theFunctionT{T}
\def\theFunctionf{f}
\def\theFunctionfZero{f_{0}}
\def\theFunctionh{h}
\def\theGAFunctionT{ \pmb{T} }
\def\certainBranchOfBoldT{ \widetilde{ \theFunctionT } }
\def\theContinuedFunctionT{ \widehat{ \theFunctionT } }
\def\theInverseOfT{ \inverseFunction{\theFunctionT} }
\def\theCurvedf{g}
\def\theCurvedfRe{\Re \theCurvedf}
\def\theCurvedfIm{\Im \theCurvedf}
\def\theCurvedH{H}
\def\theCurvedHRe{\Re H}
\def\theCurvedHIm{\Im H}

\def\theFamilyOfLevelSetsOfImT{\mathcal{G}}

\def\halfOfLogarithm{W}
\def\halfOfLogarithmRestriction{W_1}
\def\halfOfLogarithmInnerFunction{U}
\def\vCLT{\rho}
\def\maximalDomainOfGandF{\CC \setminus [-\sqrt{2+\gamma_{0}}, \sqrt{2+\gamma_{0}}]}
\def\positiveRealHalfAxis{ \RR_{+} }
\def\closedUpperHalfPlane{\topologicalClosure{ \CC_{+} }}
\def\openedFirstQuadrant{Q}
\def\semiclosedFirstQuadrant{ \topologicalClosure{ \openedFirstQuadrant } \setminus \{ 0, \sqrt{2}, \sqrt{2+\gamma_{0}} \} } 

\makeatletter
\@namedef{subjclassname@2020}{%
  \textup{2020} Mathematics Subject Classification}
\makeatother

\newtheorem{theorem}{Theorem}[section]
\newtheorem{lemma}[theorem]{Lemma}
\newtheorem{proposition}[theorem]{Proposition}
\newtheorem{corollary}[theorem]{Corollary}

\newtheorem*{theorem*}{Theorem}
\newtheorem*{lemma*}{Lemma}
\newtheorem*{proposition*}{Proposition}

\newtheorem*{corollary*}{Corollary}

\theoremstyle{definition}

\newtheorem{definition}[theorem]{Definition}
\newtheorem{remark}[theorem]{Remark}

\newtheorem*{definition*}{Definition}
\newtheorem*{remark*}{Remark}
\newtheorem*{example*}{Example}

\numberwithin{equation}{section}

\allowdisplaybreaks

\def\triExclamation{}

\begin{document}

\baselineskip=17pt

\title[V-monotone independence, CLT]{Central limit measure for V-monotone independence}

\author[A.~Dacko]{Adrian Dacko}
\address{Faculty of Pure and Applied Mathematics \newline
\indent Wroc\l{}aw University of Science and Technology \newline
\indent Wybrze\.{z}e Wyspia\'{n}skiego 27 \newline
\indent 50-370 Wroc\l{}aw \newline
\indent Poland}
\email{adrian.dacko@pwr.edu.pl}

\date{\today, \currenttime}


\begin{abstract}
We study the central limit distribution $\mu$ for V-monotone independence. Using its Cauchy--Stieltjes transform, we prove that $\mu$ is absolutely continuous with respect to the Lebesgue measure on $\RR$ and we give its density $\rho$ in an implicit form. We present a computer generated graph of $\rho$.
\end{abstract}


\maketitle



\section{Introduction}
Central limit theorems are important in all models of noncommutative independence.
In particular, in the central limit theorem for tensor independence (classical independence), free independence and Boolean independence one obtains the standard normal distribution~\cite{{Cushen&Hudson-1971}, {Giri&vonWaldenfels-1978}}, the standard Wigner distribution~\cite{{Speicher-1990}, {Voiculescu-1985}}, and the standard Bernoulli distribution~\cite{Speicher&Woroudi-1997}, respectively.
In the case of monotone (and antimonotone) independence, we obtain the standard arcsine distribution~\cite{Muraki-2001}
\begin{equation*}
\mStandardGauss(\diff x) = \mStandardGaussDensity(x) \, \diff x \text{,}
\end{equation*}
where
\begin{equation*}
\mStandardGaussDensity(x) = \frac{1}{\pi \cdot \sqrt{2 - x^2}}
\end{equation*}
with $x \in (-\sqrt{2}, \sqrt{2})$ (see also~\cite{{Lu-1997a}, {Muraki-1996}}). There are also many other noncommutative analogues of the classical central limit theorem related to other notions of independence, including interpolations and Fock space type constructions~\cite{{Bozejko-1991}, {Bozejko&Leinert&Speicher-1996}, {Bozejko&Wysoczanski-1998}, {Bozejko&Wysoczanski-2001}, {Franz&Lenczewski-1999}, {Lenczewski&Salapata-2006}, {Lenczewski&Salapata-2008}, {vonWaldenfels-1978}, {Wysoczanski-2010}}.

The notion of \mbox{V-monotone} independence was introduced and studied in~\cite{Dacko-2020}. The V-monotone independence can be considered as a combination of two twin models of independence, monotone independence and antimonotone independence, into one model. This is best seen at the Fock space level since the monotone Fock space is associated with decreasing sequences of indices assigned to direct sums, while the antimonotone Fock space is associated with increasing sequences. The simplest combination of these two families of sequences leads to the \mbox{V-monotone} Fock space, which is associated with sequences of indices whose graph has the shape of the letter V (that is, they are either increasing or decreasing, or they decrease up to a certain point and then increase), whence the name \mbox{``V-monotone''}.

In the monotone and antimonotone central limit theorems, the moments of the standard Gaussian measure $\mu$ are expressed by monotonically labeled noncrossing pair partitions and the antimonotonically labeled ones, respectively. The class of \mbox{V-monotonically} labeled noncrossing pair partitions that occurs in the \mbox{V-monotone} central limit theorem contains both of the aforementioned classes. Studying the combinatorics of this class of partitions is rather involved.
It is impossible to obtain a simple recursive formula for the moments of the limit measure $\mu$. This results in the fact that the process of obtaining $\mu$ is far more complex than in the types of independence known so far, and it is necessary to use a nonstandard approach. Eventually, we have
\begin{equation*}
\mu(\diff x) = \rho(x) \, \diff x \text{,}
\end{equation*}
where $x \in (-\sqrt{2 + \gamma_0}, \sqrt{2 + \gamma_0})$ with $\gamma_0 > 0$ defined in~\eqref{equation:partOfAnElementOfTheFrontierOfTheSupportOfTheVSGS}, and the density $\rho$ is given in an implicit form. The situation resembles that of the distribution of $T T^{*}$, where $T$ is the triangular operator, studied by Dykema and Haagerup~\cite{Dykema&Haagerup-2005}, who have also found the density function in an implicit form (see Theorem~8.9 there).

This article is devoted to determining the standard V-monotone Gaussian measure: first, we obtain its Cauchy--Stieltjes transform for real arguments, and then we extend it, in order to have it for complex arguments.
But extending one of the functions occurring in the formula for the transform is quite involved, i.e. we have to extend a certain inverse $\theInverseOfRealT$ of an elementary function. We extend $\theRealValuedT$ in Section~3, mostly working on finding its domain.
We invert this extension in Section~\ref{section:theInverseOfT}. In Section~\ref{section:CLTdensity} we obtain the Cauchy--Stieltjes transform of the V-monotone central limit distribution and then, by means of the transform, we prove that $\mu$ is absolutely continuous with respect to the Lebesgue measure on $\RR$. Finally, we get the density in an implicit form.

In the paper, we present a computer-generated graph of the density, which resembles that of the density of the standard arcsine distribution, but nevertheless differs quite significantly from it, since it has a local maximum at zero and has two symmetric singularities at $x = \pm \sqrt{2}$.

All figures presented in the paper were generated by means of the Wolfram Mathematica system.

\section{Central limit distribution}
In this section, using the results from~\cite{Dacko-2020}, we determine, first on the complement of a symmetric closed interval on $\RR$, the reciprocal Cauchy--Stieltjes transform $\theFunctionFmu$ of the measure $\mu$ obtained as a central limit distribution for V-monotone independence. Then, we analytically extend $\theFunctionFmu$ in order to have it for complex numbers. However, it is quite involved to extend one of the functions occurring in the formula for this transform and we continue to do it only in the next two sections.

Let $(m_{n})_{n=0}^\infty$ be the sequence given by
\begin{equation}
\label{equation:VmonCLT-moments}
m_{n} = 
\begin{cases}
P_{n/2}(1) & \text{ if $n$ is even, } \\
0 & \text{ if $n$ is odd, }
\end{cases}
\end{equation}
where $(P_{n})_{n=0}^{\infty}$ is the family of polynomials of variable $s \in [0, 1]$, defined recursively by
\begin{equation}
\label{equation:VmonCLT-evenMoments-polynomialsRecursion}
\left\{ \begin{aligned}
& P_{0}(s) = 1 \text{,} \\
& P_{n+1}(s) = \sum_{m=0}^{n} \Big( \int_{0}^{s} P_{m}(t) \, \diff{t} + 2^{-m} C_m (1-s)^{m+1} \Big) P_{n-m}(s) \text{,}
\end{aligned} \right.
\end{equation}
for any nonnegative integer $n$, where $(C_{n})_{n=0}^{\infty}$ is the sequence of Catalan numbers.

\begin{remark}
The sequence $(m_{n})_{n=0}^{\infty}$ has also a combinatorial interpretation (cf.~\cite[Theorem~5.3]{Dacko-2020}), which served as its definition there. The equivalence of these definitions is the result of~\cite[Theorems~6.13 and~7.5]{Dacko-2020}.

The definitions of $(P_n)_{n=0}^{\infty}$, given in~\eqref{equation:VmonCLT-evenMoments-polynomialsRecursion} and~\cite[Definition~7.4]{Dacko-2020}, are equivalent. Moreover, $\lvert P_{n}(x) \rvert \leq C_{n}$, for each $n$. Both the facts follow from the proof of~\cite[Proposition~7.6]{Dacko-2020}. It was also proved in~\cite[Theorems~6.13 and~7.5]{Dacko-2020} that $(m_{n})_{n=0}^{\infty}$ forms the moment sequence of a probability measure on $\RR$.
\end{remark}

First, let us recall the V-monotone central limit theorem~\cite[Theorem~5.3]{Dacko-2020}.
\begin{theorem}
Let $(\mathcal{A}, \varphi)$ be a $C^{*}$-probability space and let $(a_{i})_{i=1}^{\infty}$ be a family of self-adjoint, V-monotonically independent, and identically distributed random variables with mean $0$ and variance $1$. Then
\begin{equation*}
\frac{a_1 + \cdots + a_N}{\sqrt{N}} \weaklyConvergesTo \omega \text{,}
\end{equation*}
where $\omega$ is a random variable with the \emph{standard V-monotone Gaussian distribution} given by the moment sequence $(m_{n})_{n=0}^\infty$.
\end{theorem}

The distribution $\mu$ is called standard because it has mean zero and variance one. Since all odd moments vanish, $\mu$ is symmetric with respect to $0$. Its support is compact and contained in $[-2, 2]$, since its moments are dominated by the respective moments of the standard Wigner distribution.

\begin{definition}
The moment generating function of $\mu$ is given by the power series
\begin{equation}
\label{equation:vMonotoneGaussianRealMGFSeries}
\theFunctionMmu(z) = \sum_{n=0}^{\infty} m_{n} z^{n} \text{.}
\end{equation}
The Cauchy--Stieltjes transform and the reciprocal Cauchy--Stieltjes transform of this measure will be denoted by $\theFunctionGmu$ and $\theFunctionFmu$, respectively. By the definition, $\theFunctionFmu = 1 / \theFunctionGmu$.
\end{definition}

We temporarily consider $\theFunctionMmu$, $\theFunctionGmu$ and $\theFunctionFmu$ for real arguments, that is the restrictions of these functions to a suitable subsets of $\RR$. To investigate the Cauchy--Stieltjes transform, let us introduce two auxiliary functions. It was shown in~\cite{Dacko-2020} that the moment generating function of the limit distribution is expressed in terms of the function $\theRealValuedT$ given below.

Let $\theRealValuedT \colon [0, \infty) \mapsto \RR$ be given by
\begin{equation}
\label{equation:theRealVersionOfT}
\theRealValuedT(s) = -\frac{\sqrt{3}}{3} \bigg( \arctan \bigg( \frac{2s-1}{\sqrt{3}} \bigg) - \frac{\pi}{6} \bigg) - \frac{\realLogarithm( s^2 - s + 1 )}{2} \text{.}
\end{equation}
We can write this function in an integral form, namely
\begin{equation}
\label{equation:theRealVersionOfTInTheIntegralForm}
\theRealValuedT(s) = \int_{s}^{1} \dfrac{t}{t^{2} - t + 1} \, \diff{t} \text{,}
\end{equation}
for all $s \geq 0$. Indeed, simple calculations yield
\begin{equation*}
\begin{split}
\int_{s}^{1} \dfrac{t}{t^{2} - t + 1} \, \diff{t}
& = \frac{1}{2} \int_{s}^{1} \dfrac{2t - 1}{t^{2} - t + 1} \, \diff{t} + \frac{1}{2} \int_{s}^{1} \dfrac{1}{(t - \lfrac{1}{2})^2 + (\lfrac{\sqrt{3}}{2})^2} \, \diff{t} \\
& = - \frac{1}{2} \realLogarithm( s^2 - s + 1 ) -\frac{\sqrt{3}}{3} \bigg( \arctan \bigg( \frac{2s-1}{\sqrt{3}} \bigg) - \frac{\pi}{6} \bigg) = \theRealValuedT(s) \text{.}
\end{split}
\end{equation*}
The derivative of $\theRealValuedT$ is negative for $s \in \positiveRealHalfAxis \coloneqq (0, \infty)$, therefore $\theRealValuedT$ is invertible on $[0, \infty)$ and its range is $(-\infty, \theRealValuedT(0)]$, where $\theRealValuedT(0) = \lfrac{\sqrt{3} \pi}{9}$.

We define $\theInverseOfRealT \colon ( -\infty, \lfrac{\sqrt{3} \pi}{9} ] \to [0, \infty)$ as the inverse of $\theRealValuedT$. Let $\halfOfLogarithm \colon \RR \setminus [-\sqrt{2}, \sqrt{2}] \mapsto \RR$ be given by
\begin{equation}
\label{equation:halfOfLogarithmRealVersion}
\halfOfLogarithm(x) = \frac{1}{2} \realLogarithm \bigg( 1 + \dfrac{2}{x^2 - 2} \bigg) \text{}
\end{equation}
and let
\begin{equation}
\label{equation:partOfAnElementOfTheFrontierOfTheSupportOfTheVSGS}
\gamma_{0} = \frac{2}{e^{2\sqrt{3}\pi / 9} - 1} \text{.}
\end{equation}
We are now ready to write the formula for $\theFunctionFmu$ --- at this point only for real numbers.
\begin{theorem}
\label{proposition:vMonotoneGaussianRealCST}
The reciprocal Cauchy--Stieltjes transform of the standard V-monotone Gaussian measure $\mu$ is given by
\begin{equation}
\label{equation:vMonotoneGaussianRealCST}
\theFunctionFmu(x) = x \theInverseOfRealT ( W(x) ) \text{,}
\end{equation}
for $x \in \RR \setminus [-\sqrt{2+\gamma_{0}}, \sqrt{2+\gamma_{0}}]$.
\end{theorem}

\begin{proof}
The formula for the moment generating function, namely
\begin{equation}
\label{equation:vMonotoneGaussianRealMGF}
\dfrac{1}{\theFunctionMmu(x)} = \theInverseOfRealT \Bigg( \dfrac{1}{2} \realLogarithm \bigg( \dfrac{1}{1-2x^{2}} \bigg) \Bigg) \text{,}
\end{equation}
was proved in~\cite[Corollary~7.9]{Dacko-2020} for $x \in ( -\lfrac{1}{2}, \lfrac{1}{2} )$. Note that $\lfrac{1}{2} \cdot \realLogarithm \big( 1+\lfrac{2}{\gamma_{0}} ) = \lfrac{\sqrt{3} \pi}{9} = \theRealValuedT(0)$.
The function $\theInverseOfRealT$ is not analytically extendable through the point $\lfrac{\sqrt{3} \pi}{9}$. Indeed, consider the strip $\domainStripZero \coloneqq \{ u \in \CC : -\lfrac{ \sqrt{3} }{2} < \Im u < \lfrac{ \sqrt{3} }{2} \}$ and define $\certainBranchOfBoldT \colon \domainStripZero \to \CC$ as the analytic extension of $\theRealValuedT$, given by
\begin{equation*}
\certainBranchOfBoldT(u) = 
\int_{u}^{1} \dfrac{v}{v^{2} - v + 1} \, \diff{v} \text{,}
\end{equation*}
which is well defined, since the poles of the integrand do not belong to $\domainStripZero$. The point $\certainBranchOfBoldT(0) = \lfrac{\sqrt{3} \pi}{9}$ is a branch point of order two for $\inverseFunction{ \certainBranchOfBoldT }$, since $\certainBranchOfBoldT'(0) = 0$ and $0 \neq \certainBranchOfBoldT''(0) = -1$. Solving the inequality
\begin{equation*}
\dfrac{1}{2} \realLogarithm \bigg( \frac{1}{1-2x^{2}} \bigg) < \frac{\sqrt{3} \pi}{9} \text{}
\end{equation*}
with respect to $x \in (-2^{-1/2}, 2^{-1/2})$, we see that $\theFunctionMmu$ can be extended analytically to the interval $\big( -(2+\gamma_{0})^{-1/2}, (2+\gamma_{0})^{-1/2} \big)$, but not through any of its ends. Indeed, these ends correspond to $\theRealValuedT(0)$ (note that $e^{2\sqrt{3}\pi / 9} = \big(1 - 2 \cdot (2 + \gamma_{0})^{-1} \big)^{-1}$), through which $\theInverseOfRealT$ is not analytically extendable. It is well known that the moment generating function of a compactly supported probability measure on the real line has no singularities outside $\mathbbm{R}$, thus the radius of convergence of~\eqref{equation:vMonotoneGaussianRealMGFSeries} is equal to $1 / \sqrt{2+\gamma_{0}}$. Using the formula
\begin{equation*}
\theFunctionFmu(x) = \frac{x}{\theFunctionMmu(1/x)} \text{,}
\end{equation*}
we obtain our assertion.
\end{proof}

\begin{corollary}
\label{corollary:supportOfVmonotoneCLT}
The support of $\mu$ is compact and $\supp{\mu} \subseteq [-\sqrt{2+\gamma_{0}}, \sqrt{2+\gamma_{0}}]$.
\end{corollary}

Later we will show that the above inclusion is in fact an equality.
Let $\CC_{+}$ and $\CC_{-}$ denote the sets $\{ z \in \CC : \Im z > 0 \}$ and $\{ z \in \CC : \Im z < 0 \}$, respectively. One of our main goals is to determine the Cauchy--Stieltjes transform $\theFunctionGmu$.

For technical reasons, it is not convenient to write an explicit formula for $\theFunctionGmu(z)$ for every $z \in \CC_{+}$. Since $\mu$ is symmetric with respect to $0$, we have
\begin{equation}
\label{equation:csTransformOfSymmetricMeasure}
\theFunctionGmu(z) = \frac{\mu ( \{ 0 \})}{z} + 2 z \int_{\positiveRealHalfAxis} \frac{\mu(\diff t)}{z^2 - t^2} \text{,}
\end{equation}
for $z \in \maximalDomainOfGandF$, therefore $\theFunctionGmu(-\complexAdjoint{z}) = - \complexAdjoint{\theFunctionGmu(z)}$. This allows us to determine $\theFunctionGmu$ only on
\begin{equation*}
\openedFirstQuadrant \coloneqq \{ z \in \CC : \Re z > 0 \text{ and } \Im z > 0 \} \text{.}
\end{equation*}
In view of~\eqref{equation:vMonotoneGaussianRealMGF}, we expect that
\begin{equation}
\label{equation:vMonotoneGaussianComplexCST}
\theFunctionFmu(z) = z \theInverseOfT \Bigg( \dfrac{1}{2} \principalLogarithm \bigg( 1 + \dfrac{2}{z^{2}-2} \bigg) \Bigg)
\end{equation}
at least for $z \in \openedFirstQuadrant$, where $\theInverseOfT$ is a suitable analytic extension of $\theInverseOfRealT$. We will prove the above formula by showing that `$\principalLogarithm$' here is the principal value of the logarithm, i.e.
\begin{equation}
\label{equation:theSuitableLogarithm}
\principalLogarithm z = \realLogarithm R + i \varphi \text{,} \qquad \text{for $z = R e^{i \varphi}$ with $R > 0$ and $-\pi < \varphi \leq \pi$,}
\end{equation}
which is analytic on $\CC \setminus (-\infty, 0 ]$,
and by finding $\theInverseOfT$ so that $\theFunctionFmu$ is analytic on $\openedFirstQuadrant$.

We start with extending $x \mapsto \halfOfLogarithm(x)$ defined in~\eqref{equation:halfOfLogarithmRealVersion} so that, for $z \in \openedFirstQuadrant$, the extension is equal to $z \mapsto \lfrac{1}{2} \principalLogarithm ( 1 + \lfrac{2}{(z^{2} - 2)} )$, which occurs in~\eqref{equation:vMonotoneGaussianComplexCST}.
\begin{proposition}
\label{proposition:domainCutStripANDhalfOfLogarithm}
The function $\halfOfLogarithm \colon (\CC_{+} \cup \RR) \setminus \{ -\sqrt{2}, 0, \sqrt{2} \} \mapsto \CC$, given by
\begin{equation}
\label{equation:halfOfLogarithm}
\halfOfLogarithm(z) = 
\begin{cases}
\dfrac{1}{2} \principalLogarithm \bigg( 1 + \dfrac{2}{z^{2} - 2} \bigg) & \text{if $z \notin (0, \sqrt{2})$,} \\
\dfrac{1}{2} \realLogarithm \bigg \lvert 1 - \dfrac{2}{z^{2} - 2} \bigg \rvert - \dfrac{i \pi}{2} & \text{if $z \in (0, \sqrt{2})$,}
\end{cases}
\end{equation}
is a continuous extension of $x \mapsto \halfOfLogarithm(x)$ which is analytic on $\CC_{+}$. Moreover, $\halfOfLogarithm( \openedFirstQuadrant ) = \{ w \in \CC : -\lfrac{\pi}{2} < \Im w < 0 \}$ and $\halfOfLogarithm \big( (\sqrt{2+\gamma_{0}}, \infty) \big) = ( 0, \lfrac{\sqrt{3} \pi}{9} )$.
\end{proposition}

\begin{proof}
The proof of the first statement follows from elementary computations: for $z \in \CC \setminus \{ \pm \sqrt{2} \}$, let
\begin{equation}
\label{equation:halfOfLogarithmInnerFunction}
\halfOfLogarithmInnerFunction(z) = 1 + \frac{2}{z^2 - 2} \text{.}
\end{equation}
We have
\begin{equation*}
\halfOfLogarithmInnerFunction(x+iy) = 1 + \frac{2 (x^2 - y^2 - 2)}{(x^2 - y^2 - 2)^2 + 4 x^2 y^2} - \frac{4i x y}{(x^2 - y^2 - 2)^2 + 4 x^2 y^2} \text{.}
\end{equation*}
Since $\halfOfLogarithmInnerFunction(z) \in (-\infty, 0) \Leftrightarrow z \in (-\sqrt{2}, \sqrt{2}) \setminus \{ 0 \}$, it suffices to compute limits as $z = x + i y$ approaches $x_0 \in (-\sqrt{2}, \sqrt{2}) \setminus \{ 0 \}$ from above (note that $\halfOfLogarithmInnerFunction(z) \in \RR \Leftrightarrow \Re z = 0 \vee \Im z = 0$, $\lvert \Re z \rvert \neq \sqrt{2}$).
If $y > 0$ and $\lvert x - x_0 \rvert$ are small enough, then $\Im \halfOfLogarithmInnerFunction(x + i y) > 0$ and $\Im \halfOfLogarithmInnerFunction(x + i y) < 0$ for $x_0 < 0$ and $x_0 > 0$, respectively, and moreover $\Re \halfOfLogarithmInnerFunction(x + i y) < 0$. Therefore
\begin{equation*}
\lim \limits_{z \to x_0 + 0^{+} \cdot i} \frac{1}{2} \principalLogarithm \big( \halfOfLogarithmInnerFunction(z) \big) = 
\begin{cases}
\dfrac{1}{2} \realLogarithm \bigg ( \dfrac{2}{2 - x_{0}^{2}} - 1 \bigg ) + \dfrac{i \pi}{2} & \text{if $x_0 \in (-\sqrt{2}, 0)$,} \\
\dfrac{1}{2} \realLogarithm \bigg ( \dfrac{2}{2 - x_{0}^{2}} - 1 \bigg ) - \dfrac{i \pi}{2} & \text{if $x_0 \in (0, \sqrt{2})$,}
\end{cases}
\end{equation*}
and the continuity of the considered restriction of $\halfOfLogarithm$ is proved. Finally,
\begin{equation*}
\halfOfLogarithm( \openedFirstQuadrant ) = \frac{1}{2} \principalLogarithm (\CC_{-}) = \bigg\{ w \in \CC : -\frac{\pi}{2} < \Im w < 0 \bigg\} \text{}
\end{equation*}
and
\begin{equation*}
\halfOfLogarithm \big( (\sqrt{2+\gamma_{0}}, \infty) \big) = \frac{1}{2} \realLogarithm \big( (1, e^{2 \sqrt{3} \pi / 9}) \big) = \bigg( 0, \frac{\sqrt{3} \pi}{9} \bigg) \text{,}
\end{equation*}
which finishes the proof.
\end{proof}

Now, let us give the explicit form of the complex multivalued function $\theGAFunctionT$ whose branch is an analytic extension of $\theRealValuedT$. Since $\theGAFunctionT$ is expressed by means of two logarithms, we introduce two pairs of polar coordinates. Namely, let
\begin{equation*}
\label{equation:theUpperPoleOfTheIntegrand}
u_{0} = \frac{1+\sqrt{3}i}{2} \text{}
\end{equation*}
and let $u \in \CC \setminus \{ u_{0}, \complexAdjoint{u_{0}} \}$. We put
\begin{equation}
\label{equation:twoTypesOfPolarCoordinates}
r_{1} e^{i \varphi_{1}} = u - u_{0} \qquad \text{ and } \qquad r_{2} e^{i \varphi_{2}} = u - \complexAdjoint{u_{0}} \text{,}
\end{equation}
where $r_{1}$, $r_{2} > 0$ and $\varphi_{1}$, $\varphi_{2} \in \RR$.

\begin{proposition}
\label{proposition:theMultivaluedVersionOfT}
The function $\theRealValuedT$ given in~\eqref{equation:theRealVersionOfT} is a branch of a complex multivalued function $\theGAFunctionT$ given by
\begin{equation}
\label{equation:theMultivaluedVersionOfTinPolarCoordinates}
\theGAFunctionT(u) = \alpha - \beta \cdot (\realLogarithm r_{1} + i \varphi_{1}) - \complexAdjoint{\beta} \cdot (\realLogarithm r_{2} + i \varphi_{2}) \text{,}
\end{equation}
for $u \in \CC \setminus \{ u_{0}, \complexAdjoint{u_{0}} \}$, where $\alpha = -\lfrac{\sqrt{3} \pi}{9}$ and $\beta = \lfrac{(3 - \sqrt{3}i)}{6}$.
\end{proposition}

\begin{proof}
The proof is elementary: we first extend the functions occurring in~\eqref{equation:theRealVersionOfT} to the complex multivalued ones, getting
\begin{equation*}
\complexArctan{\bigg( \frac{2u-1}{\sqrt{3}} \bigg)} = \frac{\multivaluedLogarithm(u - u_{0}) - \multivaluedLogarithm(u - \complexAdjoint{u_{0}})}{2i} + \frac{\pi}{2} \text{}
\end{equation*}
and
\begin{equation*}
\multivaluedLogarithm(u^2 - u + 1) = \multivaluedLogarithm(u - u_{0}) - \multivaluedLogarithm(u - \complexAdjoint{u_{0}}) \text{,}
\end{equation*}
for $u \in \CC \setminus \{ u_{0}, \complexAdjoint{u_{0}} \}$, where $\complexArctan$ and $\multivaluedLogarithm$ are the multivalued inverse tangent function and logarithm, respectively. Elementary computations lead to
\begin{equation}
\label{equation:theMultivaluedVersionOfT}
\theGAFunctionT(u) = -\frac{\sqrt{3}\pi}{9} - \frac{3 - \sqrt{3}i}{6} \cdot \multivaluedLogarithm ( u - u_{0} ) - \frac{3 + \sqrt{3}i}{6} \cdot \multivaluedLogarithm ( u - \complexAdjoint{u_{0}} ) \text{}  
\end{equation}
Substitutions~\eqref{equation:twoTypesOfPolarCoordinates} finish the proof.
\end{proof}

\section{Function $\theFunctionT$}
\label{section:theFunctionT}
As it will be shown later, $\theFunctionFmu$ has the form
\begin{equation*}
\theFunctionFmu(z) = z \inverseFunction{T}(W(z)) \text{}
\end{equation*}
for $z \in \openedFirstQuadrant$, where $\theInverseOfT$ is the inverse of a suitable analytic extension of $\theRealValuedT$, denoted by $\theFunctionT$. We shall define $\theFunctionT$ in this section. We start with finding its domain, but in order to do so, we must first change coordinates. For our purposes, it suffices to parameterize only the upper half-plane $\CC_{+}$ (without one point). From now on, $\topologicalInterior{ A }$, $\topologicalClosure{ A }$ and $\partial A$ stand for the interior, closure and boundary of a set $A \subseteq \CC$, respectively.
\begin{definition}
\label{definition:changeOfCoordinates}
Introduce the following parametrization of $\CC_{+} \setminus \{ u_{0} \}$:
\begin{equation}
\label{equation:changeOfCoordinates0}
\begin{cases}
R = \lfrac{r_2}{r_1} \text{,} \\
\xi = \varphi_1 + \varphi_2 \text{,}
\end{cases}
\end{equation}
with $0 < r_1 < r_2$ and $\varphi_{1}, \varphi_{2} \in \RR$.
\end{definition}

It is also possible to parametrize $\CC_{-} \setminus \{ \complexAdjoint{ u_{0} } \}$ in a similar way, but we do not need it.

\begin{proposition}
\label{proposition:changeOfCoordinates}
For any $u \in \CC_{+} \setminus \{ u_{0} \}$, it holds that
\begin{equation*}
u = \theFunctionf(R, \xi) \text{,}
\end{equation*}
where $\theFunctionf \colon (1, \infty) \times \RR \to \CC$ is given by
\begin{equation}
\label{equation:changeOfCoordinates}
\theFunctionf(R, \xi) = \frac{\sqrt{3}i(R e^{i \xi} - 1)}{1 - R^2} + u_{0} \text{.}
\end{equation}
\end{proposition}

\begin{proof}
First, observe that~\eqref{equation:twoTypesOfPolarCoordinates} implies that
\begin{equation*}
r_2 e^{i \varphi_2} = u - u_{0} + (u_{0} - \complexAdjoint{ u_{0} }) = r_1 e^{i \varphi_1} + \sqrt{3} i \text{.}
\end{equation*}
Combining it with~\eqref{equation:changeOfCoordinates0}, we get the system
\begin{equation*}
\begin{cases}
R \cdot r_1 \cdot \cos(\xi - \varphi_1) = r_1 \cos \varphi_1 \text{,} \\
R \cdot r_1 \cdot \sin(\xi - \varphi_1) = r_1 \sin \varphi_1 + \sqrt{3} \text{,}
\end{cases}
\end{equation*}
which, using trigonometric identities, can be rewritten as
\begin{equation*}
\begin{cases}
(R \cos \xi - 1) \cdot r_1 \cos \varphi_1 + R \sin \xi \cdot r_1 \sin \varphi_1 = 0 \text{,} \\
R \sin \xi \cdot r_1 \cdot \cos \varphi_1 - (R \cos \xi + 1) \cdot  r_1 \sin \varphi_1 = \sqrt{3} \text{.}
\end{cases}
\end{equation*}
After substituting~\eqref{equation:twoTypesOfPolarCoordinates} and denoting $R e^{i \xi}$ and $r_1 e^{i \varphi_1} + u_{0}$ by $p + i q$ and $s + i t$, respectively, we get
\begin{equation*}
\begin{cases}
(p-1)(s - \lfrac{1}{2}) + q(t - \lfrac{\sqrt{3}}{2}) = 0 \text{,} \\
q(s - \lfrac{1}{2}) - (p+1)(t - \lfrac{\sqrt{3}}{2}) = \sqrt{3} \text{.}
\end{cases}
\end{equation*}
Since $R \neq 1$, we solve the above linear system with respect to $\tilde{s} = s - \lfrac{1}{2}$ and $\tilde{t} = t - \lfrac{\sqrt{3}}{2}$ by using Cramer's rule. This leads to
\begin{equation*}
\label{equation:changeOfCoordinates1}
s = \dfrac{-\sqrt{3} q}{1-R^2} + \dfrac{1}{2} \qquad \text{ and } \qquad t = \dfrac{\sqrt{3} (p-1)}{1-R^2} + \dfrac{\sqrt{3}}{2} \text{,}
\end{equation*}
and this is equivalent to~\eqref{equation:changeOfCoordinates}.
\end{proof}

The function $\theCurvedf$ defined below will be used to investigate the properties of $\theFunctionT$ --- a suitable branch of $\theGAFunctionT$, which is more natural than investigating these properties by means of $\theFunctionf$. Strictly speaking, the expression $\theFunctionT(\theCurvedf(\eta, \xi))$ will be studied.
\begin{definition}
\label{definition:theCurvedf}
Let $\theCurvedf \colon \domainMaximalTrapezoid \to \CC_{+} \setminus \{ u_{0} \}$, with
\begin{equation*}
\domainMaximalTrapezoid = \{ (\eta, \xi) \in \RR^2 : -\lfrac{3 \pi}{2} < \eta \leq \lfrac{\pi}{2} \text{ and } \xi < -\eta \} \text{,}
\end{equation*}
be given by
\begin{equation}
\label{equation:parametrizationOfGammaEta}
\theCurvedf(\eta, \xi) = \theFunctionf(e^{-\sqrt{3}(\eta + \xi)}, \xi) \text{.}
\end{equation}
\end{definition}

We now introduce the family of curves by means of which we will define the domain of $\theFunctionT$.
\begin{definition}
\label{definition:theFamilyOfLevelSetsOfIm}
Introduce a family of curves
\begin{equation}
\label{equation:theCurveGamma_eta}
\Gamma_\eta = \{ u = \theCurvedf(\eta, \xi) : \xi < -\eta \} \text{}
\end{equation}
parameterized by $\xi$, where $\eta \in \RR$, and let $\theFamilyOfLevelSetsOfImT = \{ \Gamma_{\eta} : -3 \pi / 2 < \eta \leq \pi / 2 \}$.
\end{definition}

\begin{remark}
We did not consider all $\eta \in \RR$ while defining $\theFamilyOfLevelSetsOfImT$ since for any $\eta \in \RR$ and any $k \in \ZZ$, it holds that
\begin{equation*}
\theCurvedf(\eta + 2 k \pi, \xi - 2 k \pi) = \theCurvedf(\eta, \xi) \text{}
\end{equation*}
and thus $\Gamma_{\eta + 2k \pi} = \Gamma_{\eta}$.
\end{remark}

Let us now give some properties of $\theCurvedf$ (cf.~Fig.~\ref{figure:hatDeltaSet}).
\begin{lemma}
\label{lemma:changeOfCoordinates-Properties}
The function $\theCurvedf$ is a continuous bijection and the inverse $\inverseFunction{ (\myRestriction{\theCurvedf}{\topologicalInterior{\domainMaximalTrapezoid}}) }$ is continuous. Moreover,
\begin{equation*}
\theCurvedf( \{ \lfrac{\pi}{2} \} \times (-\infty, -\lfrac{\pi}{2}) ) = \Gamma_{\lfrac{\pi}{2}} \qquad \text{ and } \qquad \theCurvedf(\topologicalInterior{\domainMaximalTrapezoid}) = \CC_{+} \setminus (\Gamma_{\pi / 2} \cup \{ u_{0} \}) \text{.}
\end{equation*}
The latter image is an open, simply connected set.
\end{lemma}

\begin{proof}
Consider the exterior of the unit disc
\begin{equation*}
\domainExteriorOfTheDisk = \{ v \in \CC : \lvert v \rvert > 1 \} \text{.}
\end{equation*}
In order to prove the lemma, let us represent $\theCurvedf$ as a composition
$\theCurvedf = \theFunctionfZero \circ \theFunctionh$, where $\theFunctionh \colon \domainMaximalTrapezoid \to \domainExteriorOfTheDisk$ and $\theFunctionfZero \colon \domainExteriorOfTheDisk \to \CC_{+} \setminus \{ u_{0} \}$ are defined by
\begin{equation}
\label{equation:definitionOfTheFunction_h}
\theFunctionh(\eta, \xi) = \exp \big( -\sqrt{3}(\eta + \xi) + i \xi \big) \text{}
\end{equation}
and by
\begin{equation}
\label{equation:definitionOfTheFunction_f0}
\theFunctionfZero(v) = \frac{\sqrt{3}i(v - 1)}{1 - \lvert v \rvert^2 } + u_{0} \text{,}
\end{equation}
respectively. Let us prove all necessary properties of $\theFunctionfZero$ and $\theFunctionh$. The function $\theFunctionfZero$ can be viewed as a version of $\theFunctionf$ with a complex argument (variable) if we write $v$ in polar coordinates.

\begin{figure}[ht]
\centering
\hspace*{\fill}
\begin{subfigure}{.3\textwidth}
    \centering
    \includegraphics[scale=0.15]{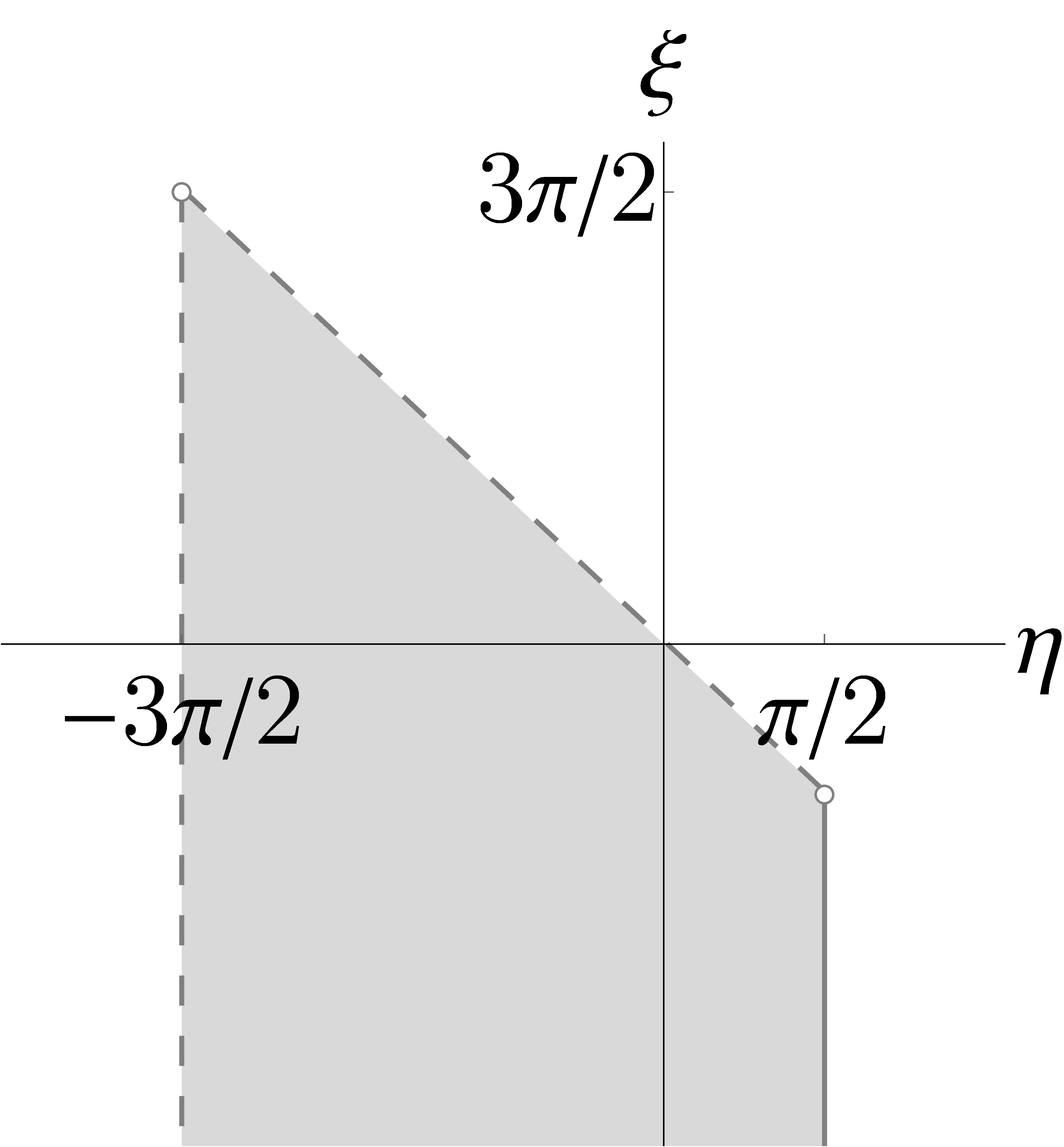}
    \label{subfigure:domainMaximalTrapezoid}
\end{subfigure} \hfill
\hspace*{\fill}
\begin{subfigure}{.3\textwidth}
    \centering
    \includegraphics[scale=0.15]{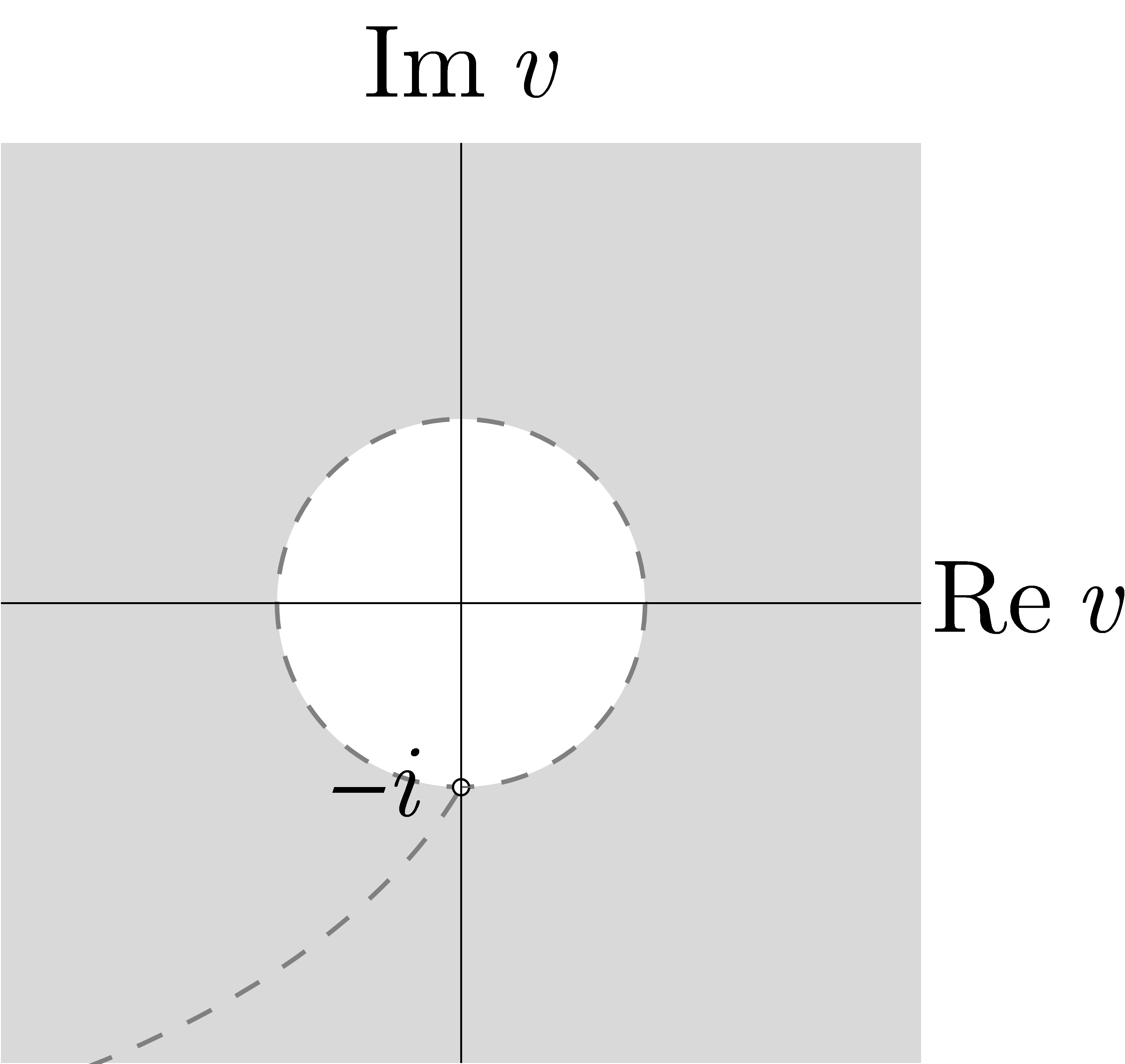}
    \label{subfigure:domainExteriorOfTheDisk}
\end{subfigure}
\hspace*{\fill}
\caption{The sets $\domainMaximalTrapezoid$ and $\domainExteriorOfTheDisk \setminus \Gamma$}
\label{figure:domainAndCodomainOfh}
\end{figure}

Let us deal with $\theFunctionh$ first, starting with injectivity. It is clear that if $(\eta_{1}, \xi_{1})$, $(\eta_{2}, \xi_{2}) \in \domainMaximalTrapezoid$
are such that $\theFunctionh(\eta_{1}, \xi_{1}) = \theFunctionh(\eta_{2}, \xi_{2})$, then
\begin{equation*}
\eta_{1} - \eta_{2} = \xi_{2} - \xi_{1} \text{}
\end{equation*}
and, since both $\xi_{1}$ and $\xi_{2}$ are arguments of $v$ (as a complex number), the $\rhs$ of the above equality is equal to $2l \pi$ for some integer $l$. Since $-2 \pi < \eta_{1} - \eta_{2} < 2 \pi$, we get $l=0$, i.e. $\theFunctionh$ is an injection. To prove its surjectivity, consider $\lvert v \rvert > 1$ and let $\eta = \eta_{0} - 2k \pi$ and $\xi = \xi_{0} + 2k \pi$, where
\begin{equation*}
\eta_{0} = -\frac{\realLogarithm \lvert v \rvert}{\sqrt{3}} - \principalArgument v \text{,} \qquad \xi_{0} = \principalArgument v \text{,} \qquad \text{and} \qquad k = \bigg \lceil \frac{\eta_{0}}{2 \pi} - \frac{1}{4} \bigg \rceil \text{.}
\end{equation*}
Here, $\principalArgument = \Im \principalLogarithm$ and, for $a \in \RR$, $\lceil a \rceil$ is the smallest integer which is $\geq a$. Note that $\eta + \xi = \eta_{0} + \xi_{0} < 0$. It suffices to demonstrate that $-\lfrac{3 \pi}{2} < \eta \leq \lfrac{\pi}{2}$ (i.e. $(\eta, \xi) \in \domainMaximalTrapezoid$) because it is clear that
\begin{equation*}
v = \exp \big( -\sqrt{3} (\eta_{0} + \xi_{0}) + i \xi_{0} \big) = \exp \big( -\sqrt{3} (\eta + \xi) + i \xi \big) \text{.}
\end{equation*}
But these inequalities follow directly from
\begin{equation*}
\frac{\eta_{0}}{2 \pi} - \frac{1}{4} \leq k < \frac{\eta_{0}}{2 \pi} + \frac{3}{4} \text{,}
\end{equation*}
which gives that $v = \theFunctionh(\eta, \xi)$ and the surjectivity of $\theFunctionh$. Its continuity is clear, i.e. it is a continuous bijection.
Let us show that $\inverseFunction{ (\myRestriction{\theFunctionh}{\topologicalInterior{ \domainMaximalTrapezoid}}) }$ is continuous. Consider a logarithmic spiral
\begin{equation}
\label{equation:theLogarithmicSpiral}
\Gamma = \{ v = \theFunctionh \big( \lfrac{\pi}{2}, \xi \big) : \xi \in \RR \} \text{.}
\end{equation}
Knowing that $\theFunctionh$ is a bijection, it is obvious that
\begin{equation*}
\theFunctionh( \domainMaximalOpenTrapezoid ) = \domainExteriorOfTheDisk \setminus \Gamma \text{.}
\end{equation*}
Let $(\eta_{n}, \xi_{n})_{n=1}^{\infty}$ be a sequence from $\domainMaximalOpenTrapezoid$ such that $v_{n} = \theFunctionh(\eta_{n}, \xi_{n}) \to \theFunctionh(\eta, \xi) = v$ with $(\eta, \xi) \in \domainMaximalOpenTrapezoid$. Let us show that $(\eta_{n}, \xi_{n}) \to (\eta, \xi)$. Consider any limit point $(\eta', \xi')$ of $(\eta_{n}, \xi_{n})_{n=1}^{\infty}$. We demonstrate that it belongs to $\domainMaximalOpenTrapezoid$ (to better understand the idea of this proof, see Fig.~\ref{figure:domainAndCodomainOfh}). If it held that $\xi' = -\infty$, then $v$ would be $\complexInfty$, which is impossible. It is also impossible that $\eta' \in \{ -\lfrac{3 \pi}{2}, \lfrac{\pi}{2} \}$, since this would imply that $v \in \Gamma$. If it held that $\xi' = -\eta'$, then $\lvert v \rvert$ would be $1$. Therefore, $(\eta', \xi') \in \domainMaximalOpenTrapezoid$ and $\theFunctionh(\eta', \xi') = \theFunctionh(\eta, \xi)$, therefore, since $\theFunctionh$ is an injection, $(\eta', \xi') = (\eta, \xi)$, and this implies that $(\eta_{n}, \xi_{n}) \to (\eta, \xi)$ and $\myRestriction{\theFunctionh}{\topologicalInterior{ \domainMaximalTrapezoid}}$ is a homeomorphism.
The function $\inverseFunction{ \theFunctionh }$ is discontinuous along $\Gamma \cap \domainExteriorOfTheDisk$. To show this, let $\xi < -\lfrac{\pi}{2}$ and put
\begin{equation*}
\domainMaximalTrapezoid \ni (\eta_{n}, \xi_{n}) = ( -\lfrac{3 \pi}{2} + \lfrac{1}{n}, \xi + 2 \pi - \lfrac{1}{n} ) \text{,} \qquad n \in \NN_{+} \text{.}
\end{equation*}
We have $v_{n} = \theFunctionh(\eta_{n}, \xi_{n}) \to \theFunctionh(\lfrac{\pi}{2}, \xi) \eqqcolon v$, but $\inverseFunction{ \theFunctionh }(v_{n}) \to (-\lfrac{3 \pi}{2}, \xi + 2 \pi) \neq (\lfrac{\pi}{2}, \xi)$.

Let us prove that $\theFunctionfZero$ is a homeomorphism. Consider the equation
\begin{equation*}
u = \theFunctionfZero(v) \text{}
\end{equation*}
of variables $v \in \domainExteriorOfTheDisk$ and $u \in \CC_{+} \setminus \{ u_0 \}$. Let us rewrite it as
\begin{equation}
\label{equation:invertingOf_f0}
u - u_0 = \frac{(u_0 - \complexAdjoint{ u_0 })(v - 1)}{1 - \lvert v \rvert^2 } \text{}
\end{equation}
and let us show
that it implies that
\begin{equation}
\label{equation:inverseOfTheFunction_f0}
v = \frac{u - \complexAdjoint{ u_0 }}{ \complexAdjoint{u} - \complexAdjoint{ u_0 }} \text{.}
\end{equation}
Indeed, recall that $\sqrt{3}i = u_{0} - \complexAdjoint{ u_{0} }$ and write
\begin{equation*}
u - \complexAdjoint{u_0} = u - u_0 + \sqrt{3}i = \frac{(u_{0} - \complexAdjoint{ u_{0} })(1 - \complexAdjoint{v})v}{1 - \lvert v \rvert^2} \text{.}
\end{equation*}
Dividing the above equation by the adjoint of~\eqref{equation:invertingOf_f0}, we obtain~\eqref{equation:inverseOfTheFunction_f0} and the injectivity of $\theFunctionfZero$. We have
\begin{equation*}
\lvert v \rvert > 1 \Leftrightarrow \lvert u - \complexAdjoint{ u_0 } \rvert > \lvert u - u_0 \rvert \Leftrightarrow u \in \CC_{+} \setminus \{ u_0 \} \text{,}
\end{equation*}
which implies that $\theFunctionfZero$ is a surjection. The continuity of $\theFunctionfZero$ and its inverse, given by~\eqref{equation:inverseOfTheFunction_f0}, is obvious.

We thus know that $\theCurvedf = \theFunctionfZero \circ \theFunctionh$ and that it is a continuous bijection. The equality
\begin{equation*}
\theCurvedf( \{ \lfrac{\pi}{2} \} \times (-\infty, -\lfrac{\pi}{2}) ) = \Gamma_{\lfrac{\pi}{2}}
\end{equation*}
follows directly from Definition~\ref{definition:theFamilyOfLevelSetsOfIm}. Of course,
\begin{equation*}
\theCurvedf( \domainMaximalOpenTrapezoid ) = \theFunctionfZero \big( \domainExteriorOfTheDisk \setminus \Gamma \big) = \CC_{+} \setminus (\Gamma_{\pi / 2} \cup \{ u_{0} \}) \text{.}
\end{equation*}
The set $\CC_{+} \setminus (\Gamma_{\pi / 2} \cup \{ u_{0} \})$ is open in $\CC$, since $(\Gamma_{\pi / 2} \cup \{ u_{0} \})$ is closed. The set $\domainMaximalOpenTrapezoid$ is simply connected and so is its image under $\theCurvedf$, which follows from the homeomorphicity of $\myRestriction{\theCurvedf}{\topologicalInterior{ \domainMaximalTrapezoid}}$, and the proof is complete.
\end{proof}

The family $\theFamilyOfLevelSetsOfImT$ has a very important property: every curve from $\theFamilyOfLevelSetsOfImT$ can be viewed as a `level-set' of $\Im \theGAFunctionT$ on $\CC_{+} \setminus \{ u_0 \}$ in the sense given below.
\begin{proposition}
\label{proposition:theLevelSetsOfT}
Let $-3 \pi / 2 < \eta \leq \pi / 2$ be fixed. For any continuous branch $\certainBranchOfBoldT$ of $\theGAFunctionT$ with a domain $\Delta_{\eta}$ such that $\Gamma_{\eta} \subseteq \Delta_{\eta} \subseteq \CC_{+} \setminus \{ u_{0} \}$, we have
\begin{equation*}
\Im \myRestriction{ \certainBranchOfBoldT }{\Gamma_{\eta}} \equiv \frac{\eta + 2k \pi}{2} \text{,}
\end{equation*}
for some $k \in \ZZ$.
\end{proposition}

\begin{proof}
Fix $\eta \in (-3 \pi / 2, \pi / 2]$ and $u \in \Gamma_{\eta}$ and let $\certainBranchOfBoldT$ be a branch of $\theGAFunctionT$ as in the formulation of the proposition. Proposition~\ref{proposition:theMultivaluedVersionOfT} and elementary computations yield
\begin{equation}
\label{equation:theImaginaryPartOfT0}
\Im \certainBranchOfBoldT(u) = - \frac{1}{2} \cdot \bigg( \frac{1}{\sqrt{3}} \realLogarithm \frac{r_{2}}{r_{1}} - (\varphi_{1} + \varphi_{2}) \bigg) \text{,} \quad u \in \Delta_{\eta} \text{,}
\end{equation}
where $0 < r_{1} < r_{2}$ and $\varphi_{1}, \varphi_{2} \in \RR$ are such that $u - u_{0} = r_{1} \cdot e^{i \varphi_{1}}$ and $u - \complexAdjoint{ u_{0} } = r_{2} \cdot e^{i \varphi_{2}}$. Let $u \in \Gamma_{\eta}$. Of course,
\begin{equation*}
u - u_{0} = \frac{\sqrt{3}i (R e^{i \xi} - 1)}{1 - R^2}
\end{equation*}
for some $\xi < -\eta$, where $R = e^{-\sqrt{3}(\eta + \xi)} > 1$. Writing $u = \theFunctionfZero(R e^{i \xi})$, from the proof of Lemma~\ref{lemma:changeOfCoordinates-Properties} we get
\begin{equation}
\label{equation:invertingParametrization}
R \cdot e^{i \xi} = \inverseFunction{ \theFunctionfZero }(u) = \frac{ r_{2} }{ r_{1} } \cdot e^{i (\varphi_{1} + \varphi_{2})} \text{.}
\end{equation}
Of course, we also know from the mentioned proof that $u \in \CC_{+} \setminus \{ u_{0} \}$. Therefore,
\begin{equation*}
R = \frac{ r_{2} }{ r_{1} } \qquad \text{ and } \qquad \xi = \varphi_{1} + \varphi_{2} - 2 k \pi \text{,}
\end{equation*} 
for some $k \in \ZZ$. Using these equalities, we can rewrite~\eqref{equation:theImaginaryPartOfT0}, obtaining
\begin{equation*}
\Im \certainBranchOfBoldT(u) = -\frac{1}{2} \bigg( \frac{\realLogarithm R}{\sqrt{3}} + \xi - 2 k \pi \bigg) = \frac{\eta + 2k \pi}{2} \text{.}
\end{equation*}
The fact that $\Im \certainBranchOfBoldT$ is continuous along $\Gamma_{\eta}$ implies that $k$ does not depend on $\xi$, which finishes the proof.
\end{proof}

It follows from Lemma~\ref{lemma:changeOfCoordinates-Properties} that
\begin{equation*}
\bigcup_{-\lfrac{3 \pi}{2} < \eta \leq \lfrac{\pi}{2}} \Gamma_{\eta} = \theCurvedf(\domainMaximalTrapezoid) = \CC_{+} \setminus \{ u_{0} \} \text{,}
\end{equation*}
which is a disjoint sum. Using the curves from $\theFamilyOfLevelSetsOfImT$, we let us define a subset of $\CC_{+} \setminus \{ u_{0} \}$ which is a domain of a suitable branch of $\theGAFunctionT$, namely
\begin{equation}
\label{equation:domainOpenUpperShell}
\domainOpenUpperShell = \bigcup \limits_{-\pi < \eta < 0} \Gamma_{\eta} \text{.}
\end{equation}
Note that $\domainOpenUpperShell = \theCurvedf( \domainOpenTrapezoid )$, where
\begin{equation*}
\domainTrapezoid \coloneqq \{ (\eta, \xi) \in \RR^2 : -\pi \leq \eta \leq 0 \text{ and } \xi < -\eta \} \text{.}
\end{equation*}
From now on, we consider only $\myRestriction{\theCurvedf}{\domainTrapezoid}$ and we will simply write $\theCurvedf$ instead of $\myRestriction{\theCurvedf}{\domainTrapezoid}$.

\begin{figure}[ht]
\centering
    \includegraphics[scale=0.3]{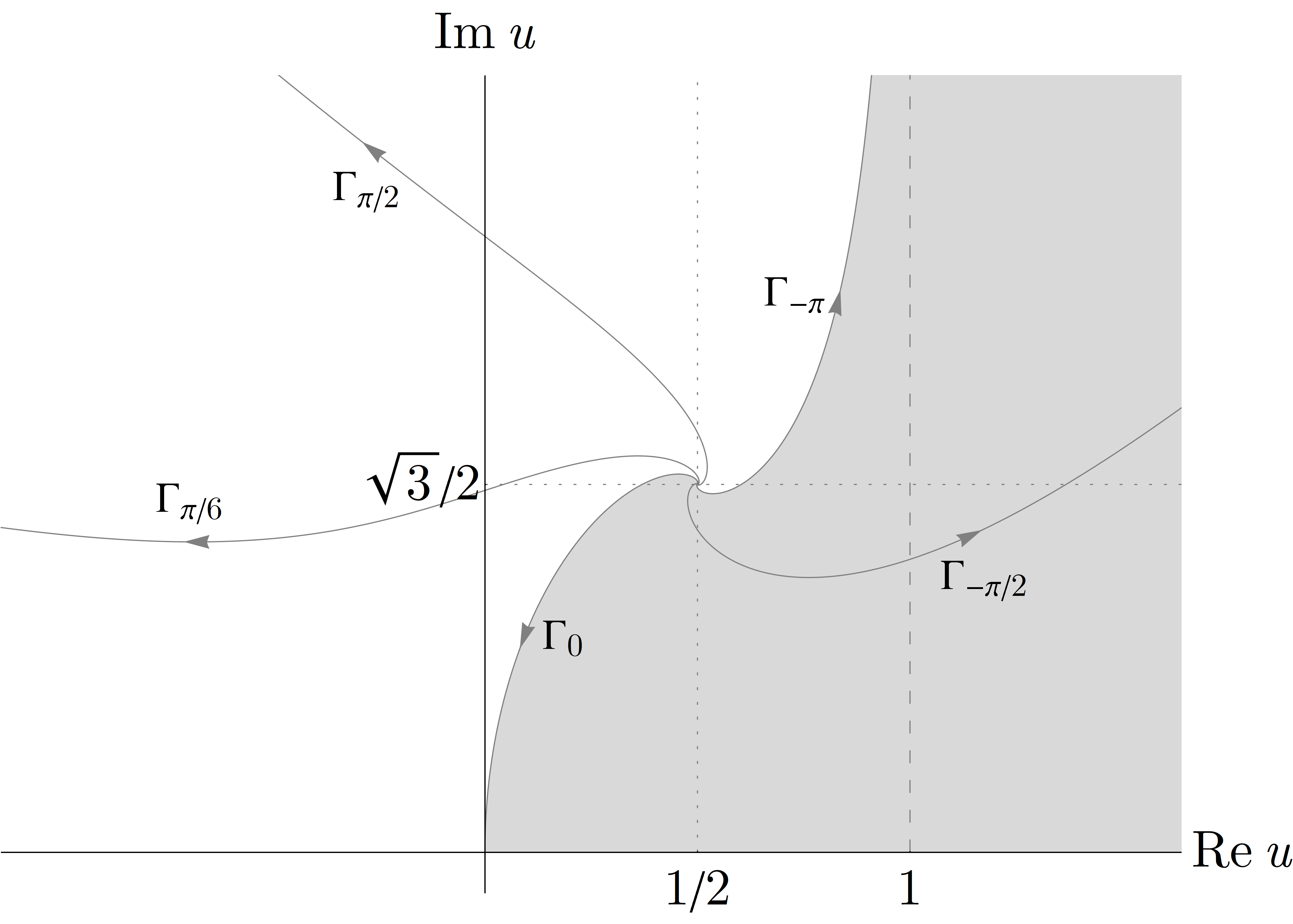}
\caption{
    The set $\domainOpenUpperShell$ bounded by $\Gamma_{0}$, $\{ u_{0} \}$, $\Gamma_{-\pi}$, and $[0, \infty)$, with sample curves $\Gamma_{-\lfrac{\pi}{2}} \subseteq \domainOpenUpperShell$ and $\Gamma_{\lfrac{\pi}{6}}$, which is disjoint with $\topologicalClosure{ \domainOpenUpperShell }$. The curve $\Gamma_{\lfrac{\pi}{2}} = \Gamma_{-\lfrac{3 \pi}{2}}$, also disjoint with $\domainOpenUpperShell$, appears in Lemma~\ref{lemma:changeOfCoordinates-Properties}. All the curves $\Gamma_{\eta}$ approach $u_{0}$ winding around it as $\xi \to -\infty$. The orientation of each $\Gamma_{\eta}$ (determined by its parametrization) is given by an arrow.
}
\label{figure:hatDeltaSet}
\end{figure}

\begin{remark}
We will not consider branches of $\theGAFunctionT$ defined on $\CC_{+} \setminus \{ u_{0} \}$, but outside $\domainOpenUpperShell$ because, due to Proposition~\ref{proposition:theLevelSetsOfT}, only a branch of $\theGAFunctionT$ with domain contained in $\domainOpenUpperShell$ (considering $\theGAFunctionT$ restricted to $\CC_{+}$ only) has its values in $\{ w \in \CC : -\pi / 2 < \Im w < 0 \}$.

Indeed, for any $\certainBranchOfBoldT$ as in Proposition~\ref{proposition:theLevelSetsOfT} with $\eta \notin (-\pi, 0)$ and any $u \in \Gamma_{\eta}$, we have $\Im \certainBranchOfBoldT(u) = \lfrac{\eta}{2} + k \pi$ for some $k \in \ZZ$, hence $\certainBranchOfBoldT(u) \notin \{ w \in \CC : -\pi / 2 < \Im w < 0 \}$.
\end{remark}

The fact that $\domainOpenUpperShell$ is a simply connected domain, proved below, proves the existence of a well-defined branch of $\theGAFunctionT$ with the domain $\domainOpenUpperShell$.
\begin{lemma}
\label{lemma:propertiesOfTheUpperShell}
The set $\domainOpenUpperShell$ is a simply connected domain with the boundary of the form
\begin{equation*}
\label{equation:boundaryOfTheUpperShell}
\partial \domainOpenUpperShell = \Gamma_{-\pi} \cup \{ u_{0} \} \cup \Gamma_{0} \cup [0, \infty) \text{.}
\end{equation*}
Moreover, $\theCurvedf(\domainTrapezoid) \subseteq \openedFirstQuadrant$.
\end{lemma}

\begin{proof}
Let $(\eta_{n}, \xi_{n})_{n=1}^{\infty}$ be a sequence of elements from $\domainTrapezoid$. For a positive integer $n$, set $u_{n} = \theCurvedf(\eta_{n}, \xi_{n})$. We start the proof by showing the following equivalences:
\begin{enumerate}
    \item[(A)] \label{item:A}
    $\xi_{n} \to -\infty \Leftrightarrow u_{n} \to u_{0}$,
    \item[(B)]
    $\lfrac{\xi_{n}}{\eta_{n}} \to (-1)^{+} \Leftrightarrow \lvert u_{n} \rvert \to \infty$
    \item[(C)]
    for any $s \geq 0$, $\big( (\eta_{n}, \xi_{n}) \to (0, 0) \wedge \lfrac{\xi_{n}}{\eta_{n}} \to \lfrac{1}{s} - 1 \big) \Leftrightarrow u_{n} \to s$ with the convention $\lfrac{1}{0} = \infty$.
\end{enumerate}
To better understand the idea of the proof, see Fig.~\ref{figure:domainTrapezoid} and Fig.~\ref{figure:hatDeltaSet}.

\begin{figure}[ht]
    \centering
    \includegraphics[scale=0.15]{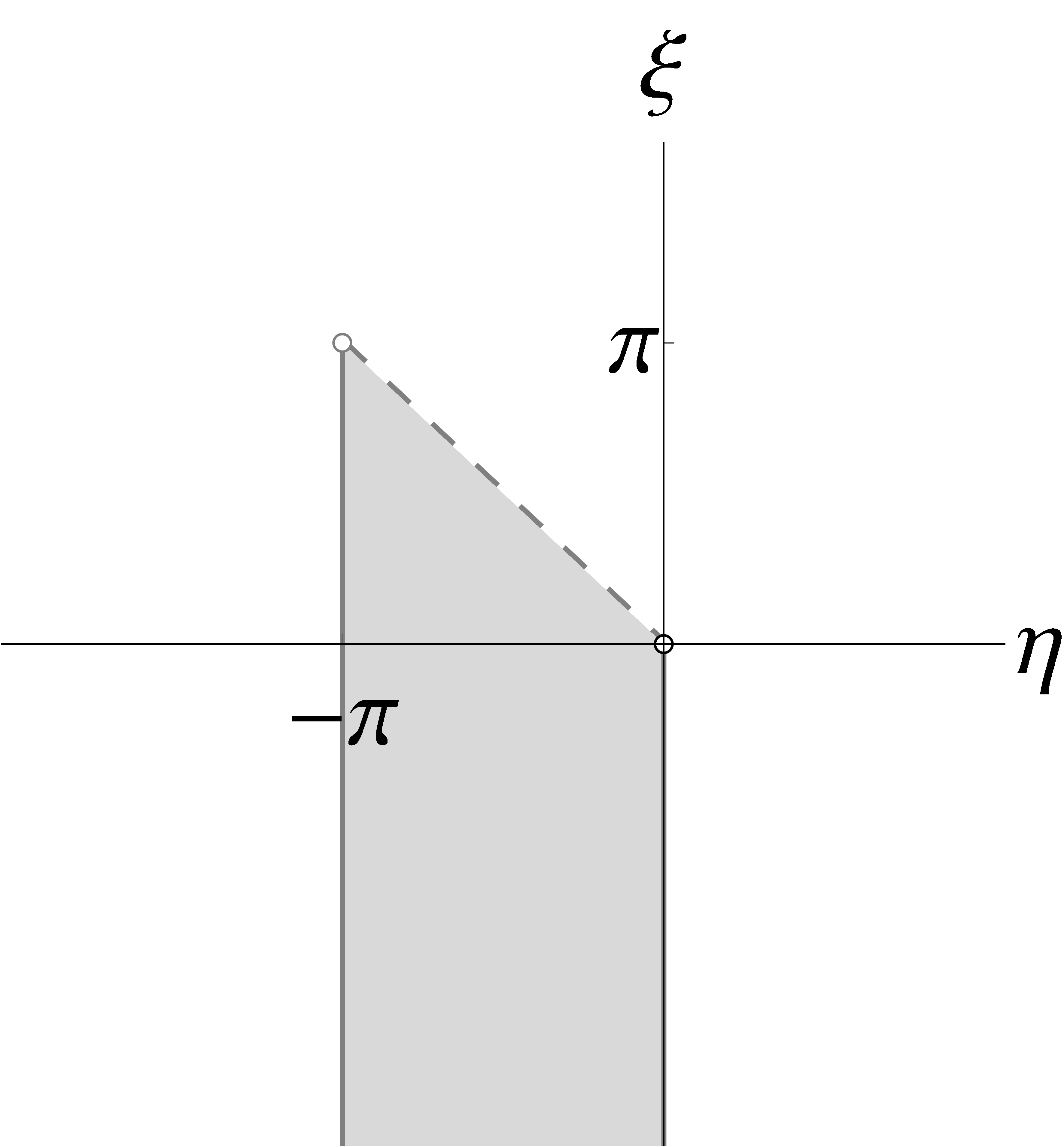}
    \caption{The set $\domainTrapezoid$}
    \label{figure:domainTrapezoid}
\end{figure}

\noindent \textit{Proof of (A).} For each $n$, let $v_{n} = \theFunctionh(\eta_{n}, \xi_{n})$ (the function $\theFunctionh$ was defined in the proof of Lemma~\ref{lemma:changeOfCoordinates-Properties}). Observe that $\xi_{n} \to -\infty$ if and only if $\lvert v_{n} \rvert \to \infty$. If $\lvert v_{n} \rvert \to \infty$, then $u_{n} \to u_{0}$. Conversely, since
\begin{equation*}
\bigg \lvert \frac{v_{n}-1}{\lvert v_{n} \rvert - 1} \bigg \rvert \geq 1 \text{,}
\end{equation*}
we have $\lvert v_{n} \rvert \to \infty$, whenever $u_{n} \to u_{0}$.

\noindent \textit{Proof of (B).}
For each $n$, we have $\xi_{n} < \pi$.
If $\xi_{n} \neq 0$, then we can write
\begin{equation}
\label{equation:behaviourOfRealAndImaginaryPartsOfw}
\begin{aligned}
\Re u_{n} & = \frac{1}{2} \cdot \frac{\sin \xi_{n}}{ \xi_{n}} \cdot \frac{\sqrt{3}(\eta_{n}+\xi_{n})}{\sinh \sqrt{3} (\eta_{n} + \xi_{n})} \cdot \bigg( \frac{ 1 - \lfrac{\xi_{n}}{\eta_{n}} }{ 1 + \lfrac{\xi_{n}}{\eta_{n}} } \bigg) + \frac{1}{2} \text{,} \\
\Im u_{n} & = \frac{1}{2} \cdot \frac{1 - \cos \xi_{n}}{ \xi_{n}^2} \cdot \frac{\sqrt{3}(\eta_{n} + \xi_{n})}{\sinh \sqrt{3} (\eta_{n} + \xi_{n})} \cdot \bigg( \frac{ 1 - \lfrac{\xi_{n}}{\eta_{n}} }{ 1 + \lfrac{\xi_{n}}{\eta_{n}} } \bigg) \cdot \xi_{n} + \bigg( \frac{\sqrt{3}}{2} - \frac{\sqrt{3}}{e^{-\sqrt{3}(\eta_{n}+\xi_{n})} + 1} \bigg) \text{.}
\end{aligned}
\end{equation}
Otherwise, we have
\begin{equation}
\label{equation:behaviourOfRealAndImaginaryPartsOfw2}
u_{n} = \frac{1}{2} + \bigg( \frac{\sqrt{3}}{2} - \frac{\sqrt{3}}{e^{-\sqrt{3}\eta_{n}} + 1} \bigg) i \text{.}
\end{equation}
If $\lfrac{\xi_{n}}{\eta_{n}} \to (-1)^{+}$, then $\xi_{n} > 0$ and $\eta_{n} \neq 0$, for $n$ large enough (of course, also $\eta_{n} + \xi_{n} < 0$).
Without loss of generality, we can consider two cases
(note that in both of them $\eta_{n} + \xi_{n} \to 0$). If $\xi_{n} \to \pi$, then
$\lfrac{\xi_{n}}{\eta_{n}} \to (-1)^{+}$ 
implies that $\Im u_{n} \to \infty$. If no subsequence of $(\xi_{n})_{n=1}^{\infty}$ tends to $\pi$, then
$\lfrac{\xi_{n}}{\eta_{n}} \to (-1)^{+}$ 
implies that $\Re u_{n} \to \infty$. Since the real-valued functions $x \mapsto \lfrac{\sin x}{x}$, $x \mapsto \lfrac{x}{\sinh x}$, and $x \mapsto \lfrac{(1 - \cos x)}{x^2}$ are bounded, also the converse is true, from~\eqref{equation:behaviourOfRealAndImaginaryPartsOfw} (of course, if $u_{n} \to \infty$, then~\eqref{equation:behaviourOfRealAndImaginaryPartsOfw2} does not hold and $\xi_{n} \neq 0$ for almost all positive $n$.)

\noindent \textit{Proof of (C).} Assume that $(\eta_{n}, \xi_{n}) \to (0, 0)$ and $\lfrac{\xi_{n}}{\eta_{n}} \to \lfrac{1}{s} - 1$, for some $s \in \RR$ (it will turn out later that $s \geq 0$). We conclude from~\eqref{equation:behaviourOfRealAndImaginaryPartsOfw} and~\eqref{equation:behaviourOfRealAndImaginaryPartsOfw2} that $u_{n} \to s$. Conversely, let $u_{n} \to s$ for some $s \geq 0$. Using the fact that~\eqref{equation:invertingOf_f0} and~\eqref{equation:inverseOfTheFunction_f0} are equivalent (which follows from the proof of Lemma~\ref{proposition:theLevelSetsOfT}), we get $v_{n} \to 1$ and thus $\eta_{n} + \xi_{n} \to 0$, by the definition of $(v_{n})_{n=1}^{\infty}$. Since $\xi_{n} - \pi \leq \eta_{n} + \xi_{n} \leq \xi_{n}$, the sequence $(\xi_{n})_{n=1}^{\infty}$ is bounded. Therefore, again by the fact that $v_{n} \to 1$, without loss of generality we assume that $\xi_{n} \to 2k \pi$, for some integer $k$. But now we also see that $0 \leq 2k \pi \leq \pi$, hence $(\eta_{n}, \xi_{n}) \to (0, 0)$. From~\eqref{equation:behaviourOfRealAndImaginaryPartsOfw} and~\eqref{equation:behaviourOfRealAndImaginaryPartsOfw2} we get $\lfrac{\xi_{n}}{\eta_{n}} \to \lfrac{1}{s} - 1$.

We find the closure of $\domainOpenUpperShell$ first, in order to find its boundary. Let $(u_{n})_{n=1}^{\infty}$ be a sequence of elements from $\domainOpenUpperShell$ which converges to some $u \in \CC$. By Lemma~\ref{lemma:changeOfCoordinates-Properties}, for each positive $n$, there exists a sequence $(\eta_{n}, \xi_{n})_{n=1}^{\infty}$ of elements from $\domainOpenTrapezoid$ such that $u_{n} = \theCurvedf(\eta_{n}, \xi_{n})$. Without loss of generality, we assume that $(\eta_{n}, \xi_{n})_{n=1}^{\infty}$ has a limit. If $\xi_{n} \to -\infty$, then $u = u_{0}$, by~(A). Otherwise, we denote the limit of $(\eta_{n}, \xi_{n})_{n=1}^{\infty}$ by $(\eta, \xi)$). If this limit belongs to $\domainTrapezoid$, then from the homeomorphicity of $\theCurvedf$, we have
\begin{equation}
\label{eqation:theRangeOfTheCurvedf}
u \in \theCurvedf( \domainTrapezoid ) = \bigcup \limits_{-\pi \leq \eta \leq 0} \Gamma_{\eta} = \Gamma_{-\pi} \cup \domainOpenUpperShell \cup \Gamma_{0} \text{.}
\end{equation}
If $(\eta, \xi) \notin \domainTrapezoid$, then $\eta_{n} + \xi_{n} \to 0$ (i.e. the sequence $(\eta_{n}, \xi_{n})$ approaches one of the edges of the trapezoid $\domainTrapezoid$).
By~(B), it is impossible that $\lfrac{\xi_{n}}{\eta_{n}} \to (-1)^{+}$. Again, without loss of generality, we assume that $\lim_{n \to \infty} \lfrac{\xi_{n}}{\eta_{n}} \to (-1)^{+} = \lfrac{1}{s} - 1$, for some $s \in \RR$.
For all positive $n$, $(\eta_{n}, \xi_{n}) \in \domainTrapezoid$, hence $0 < \lfrac{\eta_{n}}{(\eta_{n} + \xi_{n})} \to s \geq 0$. This proves that
\begin{equation}
\label{equation:topologicalClosureOfUpperShell}
\topologicalClosure{\domainOpenUpperShell} = \Gamma_{-\pi} \cup \{ u_{0} \} \cup \domainOpenUpperShell \cup \Gamma_{0} \cup [0, \infty) \text{.}
\end{equation}

From Lemma~\ref{lemma:changeOfCoordinates-Properties} and the fact that $\domainTrapezoid \subseteq \domainMaximalOpenTrapezoid$ is simply connected, the set $\theCurvedf( \domainOpenTrapezoid ) = \domainOpenUpperShell$ is open and simply connected in $\CC$. Therefore,
\begin{equation*}
\partial \domainOpenUpperShell = \topologicalClosure{\domainOpenUpperShell} \setminus \domainOpenUpperShell = \Gamma_{-\pi} \cup \{ u_{0} \} \cup \Gamma_{0} \cup [0, \infty) \text{,}
\end{equation*}
which finishes the proof of the first statement.

Let us deal with the second one. Since we already know (from the proof of Lemma~\ref{lemma:changeOfCoordinates-Properties}) that $\Im u = \Im \theCurvedf(\eta, \xi) > 0$ for $(\eta, \xi) \in \domainTrapezoid$, it suffices to show that $\Re \theCurvedf(\eta, \xi) > 0$. Let $R = e^{-\sqrt{3}(\eta+\xi)}$.
We have
\begin{equation*}
\Re \theCurvedf(\eta, \xi) = \frac{R^2 + 2 \sqrt{3} R \sin \xi - 1}{2R^2 - 2} = \frac{(R - \alpha)(R-\beta)}{2R^2 - 2} \text{,}
\end{equation*}
where
\begin{equation*}
\alpha = \sqrt{3 \sin^2 \xi + 1} - \sqrt{3} \sin{\xi} \qquad \text{and} \qquad \beta = -\sqrt{3 \sin^2 \xi + 1} - \sqrt{3} \sin{\xi} \text{.}
\end{equation*}
Since $\alpha > 0 > \beta$ and $R > 1$, what we must prove is equivalent to
\begin{equation}
\label{equation:inequality0}
R > \sqrt{3 \sin^2 \xi + 1} - \sqrt{3} \sin{\xi} \text{.}
\end{equation}
We prove this inequality by proving the following equivalent inequality instead:
\begin{equation}
\label{equation:inequality1}
e^{-\sqrt{3} (\eta+\xi)} > \sqrt{3 \sin^2 \xi + 1} - \sqrt{3} \sin{\xi} \text{}
\end{equation}
for $(\eta, \xi) \in \domainTrapezoid$. Consider two cases. First, if $\xi \geq 0$, then $\xi < -\eta < \pi$, therefore $\rhs \leq 2 - \sqrt{3}$. Since $\lhs \geq 1$, the inequality is fulfilled. It remains to show~\eqref{equation:inequality1} for $\xi < 0$ and it is clear that we can assume that $\eta = 0$. We split this case into two subcases: $\xi \in (-\pi, 0)$ and $\xi \leq -\pi$. In the first one we put $a = -\sqrt{3} \sin \xi > 0$ and obtain
\begin{equation*}
e^{-\sqrt{3} \xi} > e^a > 1 + a + \tfrac{1}{2} a^2 > \sqrt{a^2 + 1} + a \text{,}
\end{equation*}
which is our assertion. In the remaining subcase we have
\begin{equation*}
e^{-\sqrt{3} \xi} \geq e^{\sqrt{3} \pi} > 2 + \sqrt{3} \geq \sqrt{3 \sin^2 \xi + 1} - \sqrt{3} \sin{\xi} \text{.}
\end{equation*}
This finishes the proof.
\end{proof}

Now, we are ready to define the function $\theFunctionT$.
\begin{definition}
\label{definition:theSuitableExtensionOfT}
Let $\theFunctionT$ be a continuous extension of $\theRealValuedT$ (given in~\eqref{equation:theRealVersionOfT}), to $\topologicalClosure{\domainOpenUpperShell} \setminus \{ u_{0} \}$ that is analytic on $\domainOpenUpperShell$. It follows from Lemma~\ref{lemma:propertiesOfTheUpperShell} and from the integral representation of a branch of the logarithmic function that $\theFunctionT$ is well-defined and unique.
\end{definition}

\begin{remark}
\label{remark:hatDeltaSet}
The set $\domainOpenUpperShell$ is presented in Fig.~\ref{figure:hatDeltaSet}. As we will show later, in Theorem~\ref{theorem:cauchyTransformOfVmonotoneStandardGaussian}, the Cauchy--Stieltjes transform $\theFunctionGmu$ can be written as
\begin{equation*}
\theFunctionGmu(z) = \frac{1}{z u} \text{, } \qquad \text{ with } u = \theInverseOfT(\halfOfLogarithm(z)) \text{.}
\end{equation*}
There is a correspondence between the sets $\partial \openedFirstQuadrant$ and $\partial \domainOpenUpperShell$, associated with this form of the transform (we now consider both $\openedFirstQuadrant$ and $\domainOpenUpperShell$ as a subsets of the Riemann sphere). The fact that $z \in \openedFirstQuadrant$ approaches certain subsets of $[0, \infty)$ is equivalent to the fact that $u \in \domainOpenUpperShell$ approaches respective subsets of $\partial \domainOpenUpperShell$, which is described by Table~\ref{table:theBoundaryOfUpperShell} (see also Fig.~\ref{figure:hatDeltaSet}) and demonstrated in the proof of Theorem~\ref{theorem:vMonotoneStandardGaussianDistribution}.

\begin{table}[h!]
    \centering
    \caption{The correspondence between $\partial \openedFirstQuadrant$ and $\partial \domainOpenUpperShell$}
    \renewcommand{\arraystretch}{1.5}
    \begin{tabular}{  c|c c c c c c c c  }
        \hline
        $z$ & $0$ & $(0, \sqrt{2})$ & $\sqrt{2}$ & $(\sqrt{2}, \sqrt{2+\gamma_{0}})$ & $\sqrt{2+\gamma_{0}}$ & $(\sqrt{2+\gamma_{0}}, \infty)$ & $\complexInfty$ & $i \cdot \positiveRealHalfAxis$ \\
        $u$ & $\complexInfty$ & $\Gamma_{-\pi}$ & $u_{0}$ & $\Gamma_{0}$ & $0$ & $(0, 1)$ & $1$ & $(1, \infty)$ \\ 
        \hline
    \end{tabular}
    \label{table:theBoundaryOfUpperShell}
    \renewcommand{\arraystretch}{1}
\end{table}

In addition, the last column describes the correspondence between the imaginary upper half-axis and certain subset of $\partial \domainOpenUpperShell$. This particular fact follows from Proposition~\ref{proposition:domainCutStripANDhalfOfLogarithm} and from~(C*) from the proof of Lemma~\ref{lemma:homeomorphicityOfH}.
Finally, the last but one column describes the behavior of $u$ as $z \to \complexInfty$. This limit can be determined directly using the fact that $\theInverseOfT$ is continuous.
\end{remark}

\section{Function $\theInverseOfT$}
\label{section:theInverseOfT}
In this section we obtain the inverse of $\theFunctionT$, which is a continuous extension of $\theInverseOfRealT$ and which is analytic on the interior of its domain. We prove the existence of $\theInverseOfT$ and we find its implicit form. In order to do so, we first introduce an auxiliary function $\theCurvedH$ and we study its properties.
\begin{definition}
\label{definition:theFunctionH}
Let $\theCurvedH \colon \domainTrapezoid \to \CC$ be defined by
\begin{equation}
\label{equation:theFunctionH}
\theCurvedH(\eta, \xi) = \theFunctionT \big( \theCurvedf(\eta, \xi) \big) \text{.}
\end{equation}
The fact that this function is well defined follows from the inclusion
\begin{equation*}
\theCurvedf( \domainTrapezoid ) = \bigcup_{-\pi \leq \eta \leq 0} \Gamma_{\eta} \subseteq \topologicalClosure{ \domainOpenUpperShell } \setminus \{ u_{0} \} \text{,}
\end{equation*}
which is implied by~\eqref{equation:topologicalClosureOfUpperShell}.
\end{definition}

In the following three statements we prove that $\theCurvedH$ is invertible and that its inverse is continuous. We find the exact form and the range of $\theCurvedH$ in the process.
\begin{proposition}
\label{proposition:injectivityOfH}
The function $\theCurvedH$ is an injection.
\end{proposition}

\begin{proof}
First, according to Proposition~~\ref{proposition:theLevelSetsOfT} and the continuity of $\theCurvedH$, there exists $n_{1} \in \ZZ$ (independent of $\xi$) such that
\begin{equation*}
\theCurvedHIm(\eta, \xi) = \frac{\eta + 2 n_{1} \pi}{2} \text{}
\end{equation*}
for any $(\eta, \xi) \in \domainTrapezoid$. The quantity $n_{1}$ does not depend also on $\eta$.
Otherwise, it would have to be a constant function, since it is continuous, integer-valued, and defined on a connected set (we will determine $n_{1}$ soon).
For that reason, it suffices to show that, for each $\eta \in [-\pi, 0]$, the function $\xi \mapsto \theCurvedHRe(\eta, \xi)$ is strictly monotonic. For any real $s$ and $t$ such that $s + i t \in \topologicalClosure{\domainOpenUpperShell} \setminus \{ u_{0} \}$, define the real-valued functions $\theFunctionT_{1}$ and $\theFunctionT_{2}$ by
\begin{equation*}
\theFunctionT_{1}(s, t) + i \theFunctionT_{2}(s, t) = \theContinuedFunctionT(s+it) \text{,}
\end{equation*}
where $\theContinuedFunctionT$ is the analytical continuation of $\myRestriction{\theFunctionT}{\domainOpenUpperShell}$ to $\CC_{+} \setminus (\Gamma_{\pi / 2} \cup \{ u_{0} \})$. The existence of such a continuation is provided by Lemma~\ref{lemma:changeOfCoordinates-Properties} and
the integral representation of a branch of the logarithmic function. Now, using methods of the analytic geometry, we demonstrate that $(\partial / \partial \xi) \theCurvedHRe$ has a constant sign with respect to $\xi$. Fix $\eta \in [-\pi, 0]$. It is quite obvious that $\xi \mapsto \theCurvedf(\eta, \xi)$ is a $\mathcal{C}^{\infty}$-function. Using the chain rule, we get
\begin{equation*}
\frac{\partial}{\partial \xi} \theCurvedHRe = \frac{\partial \theFunctionT_{1}}{\partial s} \cdot \frac{\partial}{\partial \xi} \theCurvedfRe + \frac{\partial \theFunctionT_{1}}{\partial t} \cdot \frac{\partial}{\partial \xi} \theCurvedfIm = \langle \nabla \theFunctionT_{1}, \vec{u}_{\xi} \rangle \text{,}
\end{equation*}
where
\begin{equation*}
\vec{u}_{\xi} \coloneqq \frac{\partial}{\partial \xi} ( \theCurvedfRe, \theCurvedfIm ) \text{}
\end{equation*}
(for convenience, we dropped the arguments of the functions).
By Proposition~\ref{proposition:theLevelSetsOfT} we have
\begin{equation*}
0 = \frac{\partial}{\partial \xi} \theCurvedHIm = \frac{\partial \theFunctionT_{2}}{\partial s} \cdot \frac{\partial}{\partial \xi} \theCurvedfRe + \frac{\partial \theFunctionT_{2}}{\partial t} \cdot \frac{\partial}{\partial \xi} \theCurvedfIm \text{,}
\end{equation*}
i.e. $\nabla \theFunctionT_{2} \perp \vec{u}_{\xi}$. From the Cauchy--Riemann equations, also $\nabla \theFunctionT_{1} \perp \nabla \theFunctionT_{2}$.

In order to prove that $\xi \mapsto \theCurvedHRe(\eta, \xi)$ is strictly monotonic,
it remains to show that $\nabla \theFunctionT_{1} \parallel \vec{u}_{\xi}$. For this purpose, we need to demonstrate that all $\nabla \theFunctionT_{1}$, $\nabla \theFunctionT_{2}$ and $\vec{u}_{\xi}$ are nonzero vectors. Indeed, for any suitable $u$, we have
\begin{equation*}
\theContinuedFunctionT'(u) = \frac{u}{u^2 - u +1} \text{,}
\end{equation*}
which is nonzero on $\CC_{+} \setminus (\Gamma_{\pi / 2} \cup \{ u_{0} \})$, therefore both $\nabla \theFunctionT_{1}$ and $\nabla \theFunctionT_{2}$ are nonzero vectors.
In order to show that also $\vec{u}_{\xi} \neq (0, 0)$, let
\begin{equation*}
p + i q = R e^{i \xi} \text{,} \qquad \text{ with } R = e^{-\sqrt{3}(\eta + \xi)} \text{,}
\end{equation*}
for any $(\eta, \xi) \in \RR^{2}$. Using~\eqref{equation:changeOfCoordinates} and~\eqref{equation:parametrizationOfGammaEta}, we get
\begin{equation*}
\frac{\partial \theCurvedf}{\partial \xi}(\eta, \xi) = \frac{-\sqrt{3} \big( p(1 - R^2) - \sqrt{3}q(1+R^2) \big) + 6i \big( R^2 - 1/2 p(1 + R^2) - \sqrt{3} / 6 q(1-R^2) \big)}{(1 - R^2)^2}
\end{equation*}
The above derivative is zero if and only if the following system is fulfilled:
\begin{equation*}
\begin{cases}
p(1 - R^2) - \sqrt{3}q(1+R^2) & = 0 \text{,} \\
R^2 - 1/2 p(1 + R^2) - \sqrt{3} / 6 q(1-R^2) & = 0 \text{.}
\end{cases}
\end{equation*}
If we multiply the latter equation by $-2p$ and then we replace $p(1-R^2)$ by $\sqrt{3}q(1+R^2)$ in its third summand, we obtain
\begin{equation*}
R^2 \cdot ( (p-1)^2 + q^2 ) = 0 \text{,}
\end{equation*}
which gives two solutions: $(p, q) = (0, 0)$ and $(p, q) = (1, 0)$, corresponding to $\eta + \xi = \infty$ and $\xi = - \eta$, respectively.
None of these solutions belongs to $\domainMaximalOpenTrapezoid$, therefore $\vec{u}_{\xi} \neq (0, 0)$. Thus we get $\nabla \theFunctionT_{1} \parallel \vec{u}_{\xi}$, i.e. $(\partial / \partial \xi) \theCurvedHRe$ has a constant nonzero sign, which means that $\xi \mapsto \theCurvedHRe(\eta, \xi)$ is strictly monotonic and therefore $\theCurvedH$ is indeed an injection.
\end{proof}

In order to find the range of $\theCurvedH$, we first obtain its exact form.
\begin{lemma}
\label{lemma:exactFormOfH}
For any $(\eta, \xi) \in \domainTrapezoid$, we have
\begin{equation}
\label{equation:theExactFormOfTheCurvedH}
\begin{split}
\theCurvedH(\eta, \xi) =
& - \frac{\sqrt{3}}{6} \bigg( \principalArgument(R e^{i \xi} - 1) - \principalArgument(u - \complexAdjoint{ u_{0} }) + 2 \bigg \lceil \frac{\xi - \pi}{2 \pi} \bigg \rceil \pi + \frac{ \pi}{6} \bigg) \\
& \ - \frac{\realLogarithm ( \lvert u - u_{0} \rvert \cdot \lvert u - \complexAdjoint{ u_{0} } \rvert )}{2} + \frac{i \eta}{2} \text{,}
\end{split}
\end{equation}
where $R = e^{-\sqrt{3}(\eta + \xi)}$ and $u = \theCurvedf(\eta, \xi)$.
\end{lemma}

\begin{proof}
First, according to~\eqref{equation:theMultivaluedVersionOfTinPolarCoordinates}, we have 
\begin{equation*}
\theFunctionT(u) = -\frac{\realLogarithm ( r_{1} r_{2} )}{2} - \frac{\sqrt{3}(\varphi_{1} - \varphi_{2} +  2 \pi / 3)}{6} - i \cdot \bigg( \frac{\sqrt{3}}{6} \realLogarithm \frac{r_{2}}{r_{1}} + \frac{\varphi_{1} + \varphi_{2}}{2} \bigg) \text{.}
\end{equation*}
We have already obtained $\theCurvedHIm(\eta, \xi)$ in the proof of Proposition~\ref{proposition:injectivityOfH}. Now, let us determine $\theCurvedHRe(\eta, \xi)$ with $(\eta, \xi) \in \domainTrapezoid$. Of course,
\begin{equation*}
\theCurvedHRe(\eta, \xi) = -\frac{\realLogarithm ( \lvert u - u_{0} \rvert \cdot \lvert u - \complexAdjoint{ u_{0} } \rvert )}{2} - \frac{\sqrt{3}}{6} \cdot \bigg( \certainArgumentB \Big( \frac{ u - u_{0} }{ u - \complexAdjoint{ u_{0} } } \Big) +  \frac{2 \pi}{3} \bigg) \text{,}
\end{equation*}
where `$\certainArgumentB$' is a certain branch of the argument and where $u = \theCurvedf(\eta, \xi)$. We must thus determine
\begin{equation*}
(\eta, \xi) \mapsto \certainArgumentB \bigg( \frac{ u - u_{0} }{ u - \complexAdjoint{ u_{0} } } \bigg) \text{,}
\end{equation*}
which will be done by means of `$\principalArgument$', i.e. the principal branch of the argument. In other words, $\principalArgument = \Im \principalLogarithm$, where `$\principalLogarithm$' was defined by~\eqref{equation:theSuitableLogarithm}. As before, let $R = e^{\sqrt{3}(\eta + \xi)}$. By the definition of $\theCurvedf$ and since $R > 1$, we can write
\begin{equation*}
\certainArgumentB \Big( \frac{ u - u_{0} }{ u - \complexAdjoint{ u_{0} } } \Big) = \principalArgument(R e^{i \xi} - 1) - \principalArgument(u - \complexAdjoint{ u_{0} }) + 2k(\eta, \xi)\pi + 3 / 2 \pi \text{,}
\end{equation*}
where $k$ is some integer-valued function, which we now deal with. Since $u \in \CC_{+}$, the function $(\eta, \xi) \mapsto \principalArgument(u - \complexAdjoint{ u_{0} })$ is continuous and has values in $[0, \pi]$. For each $n \in \ZZ$, the function
\begin{equation}
\label{equation:summandInCertainArgumentB}
(\eta, \xi) \mapsto \principalArgument(R e^{i \xi} - 1) \text{,} \qquad \eta + \xi < 0
\end{equation}
is continuous on the infinite trapezoid
\begin{equation*}
\{ (\eta, \xi) \in \RR^{2} : \zeta_{n-1} < \xi \leq \zeta_{n} \text{ and } \eta < -\xi \} \text{,}
\end{equation*}
where $\zeta_{n} = (2n+1) \pi$, and is discontinuous on the half-lines $(-\infty, -\zeta_{n}) \times \{ \zeta_{n} \}$, for every $n \in \ZZ$.
Indeed, the function~\eqref{equation:summandInCertainArgumentB} is discontinues if and only if $R e^{i \xi} \in (-\infty, 1) + \{ 0 \} \cdot i$ but since $R > 1$, this condition is equivalent to $R e^{i \xi} \in (-\infty, 0) + \{ 0 \} \cdot i$.
Fix $-\pi \leq \eta \leq 0$. The function $k(\eta, \cdot)$ has to be constant on $(\zeta_{n-1}, \zeta_{n}]$ and equal to, say, $k_{n}$. Comparing the left- and right-sided limits of $\varphi_{1} - \varphi_{2}$ as $\xi$ approaches $\zeta_{n}$, we get the following recurrence:
\begin{equation*}
k_{n+1} = k_{n} + 1 \text{.}
\end{equation*}
Therefore $k_{n} = n_{2}(\eta) + n$ for some integer $n_{2}(\eta)$, which implies that
\begin{equation*}
k(\eta, \xi) = n_{2}(\eta) + \bigg \lceil \frac{\xi - \pi}{2 \pi} \bigg \rceil \text{.}
\end{equation*}
The function $(\eta, \xi) \mapsto \principalArgument(R e^{i \xi} - 1) + 2 (n_{2}(\eta) + \lceil (\xi - \pi) / 2\pi \rceil) \pi$ is continuous. If we fix $\xi = -2 \pi$ and consider this function as function of one variable $\eta$, we get the (continuous!) function of the form $\eta \mapsto 2 n_{2}(\eta) \pi$, which means that $n_{2}(\eta) = n_{2} \in \ZZ$ does not depend on $\eta$. Therefore
\begin{equation*}
\label{equation:theExactFormOfArgumentOfQuotientWithConstants}
\begin{split}
\theCurvedH(\eta, \xi) =
& -\frac{\sqrt{3}}{6} \bigg( \principalArgument(R e^{i \xi} - 1) - \principalArgument(u - \complexAdjoint{ u_{0} }) + 2 n_{2} \pi + 2 \bigg \lceil \frac{\xi - \pi}{2 \pi} \bigg \rceil \pi + \frac{13 \pi}{6} \bigg) \\
& \ -\frac{\realLogarithm ( \lvert u - u_{0} \rvert \cdot \lvert u - \complexAdjoint{ u_{0} } \rvert )}{2} + i \cdot \frac{\eta + 2 n_{1} \pi}{2} \text{}
\end{split}
\end{equation*}
for some integers $n_{1}$ and $n_{2}$ which we now compute.

Let $(\eta_{n}, \xi_{n})_{n=1}^{\infty}$ be a sequence of elements from $\domainTrapezoid$. For each $n$, let $u_{n}$ be the same as in the proof of Lemma~\ref{lemma:propertiesOfTheUpperShell}, i.e. $u_{n} = \theCurvedf(\eta_{n}, \xi_{n})$, and let $R_{n} = e^{-\sqrt{3} (\eta_{n} + \xi_{n})}$. First, note that
\begin{equation}
\label{equation:theLimitOfEsEnPrim}
\lim \limits_{n \to \infty} \frac{\eta_{n}}{\eta_{n} + \xi_{n}} = \lim \limits_{n \to \infty} \Bigg( \frac{\sqrt{3}}{2} \cdot \frac{\xi_{n}}{R_{n} - 1} + \frac{1}{2} \Bigg) \text{}
\end{equation}
(the existence of the former limit is equivalent to the existence of the latter one), since
\begin{equation*}
\frac{\sqrt{3}}{2} \cdot \frac{\xi_{n}}{\realLogarithm R_{n}} = \frac{1}{2} \cdot \frac{\eta_{n} - \xi_{n}}{\eta_{n} + \xi_{n}} \cdot \frac{R_{n} - 1}{\realLogarithm R_{n}} \text{.}
\end{equation*}
Let us assume that $(\eta_{n}, \xi_{n}) \to (0, 0)$ and $\lfrac{\xi_{n}}{\eta_{n}} \to \lfrac{1}{s} - 1$ for some $s \in \RR$. Here $s$ has to be nonnegative, since $\lfrac{\eta_{n}}{(\eta_{n} + \xi_{n})} > 0$, which follows from the proof of the mentioned lemma. From~(C) of the same proof, it is equivalent to the fact that $u_{n} \to s$. Let us now compute $\lim_{n \to \infty} w_{n}$, where $w_{n} = \theCurvedH(\eta_{n}, \xi_{n})$. By Definitions~\ref{definition:theSuitableExtensionOfT} and~\ref{definition:theFunctionH}, this limit is equal to $\theRealValuedT(s)$. By the definition of $s$ and~\eqref{equation:theLimitOfEsEnPrim}, we have $\xi_{n} / (R_{n} - 1) \to (2s-1) / \sqrt{3}$. Note that
\begin{equation*}
\begin{split}
\principalArgument( R_{n} e^{i \xi_{n}} - 1 ) & - \principalArgument(u_{n} - \complexAdjoint{ u_{0} }) = \principalArgument \bigg( \frac{R_{n} e^{i \xi_{n}} - 1}{R_{n} - 1} \bigg) - \principalArgument(u_{n} - \complexAdjoint{ u_{0} }) \\
& = \principalArgument \bigg( R_{n} \cdot \frac{\xi_{n}}{R_{n}-1} \cdot \frac{e^{i \xi_{n}} - 1}{\xi_{n}} + 1 \bigg) - \principalArgument(u_{n} - \complexAdjoint{ u_{0} }) \\
& \to \principalArgument \bigg( 1 + \frac{2s-1}{\sqrt{3}} \cdot i \bigg) - \principalArgument \bigg( s - \frac{1}{2} + \frac{\sqrt{3}i}{2} \bigg) = 2 \arctan \bigg( \frac{2s-1}{\sqrt{3}} \bigg) - \frac{\pi}{2} \text{}
\end{split}
\end{equation*}
and, since $\xi_{n} \to 0$, we have $\lceil (\xi_{n} - \pi) / 2 \pi \rceil \to 0$. Therefore
\begin{equation*}
w_{n} \to -\frac{\sqrt{3}}{3} \bigg( \arctan \bigg( \frac{2s-1}{\sqrt{3}} \bigg) + n_{2} \pi - \frac{\pi}{6} \bigg) - \frac{\realLogarithm( s^2 - s + 1 )}{2} + i n_{1} \pi = \theRealValuedT(s) - \frac{n_{2} \pi}{\sqrt{3}} + i n_{1} \pi \text{,}
\end{equation*}
hence $n_{1} = n_{2} = 0$, which establishes the desired formula.
\end{proof}

Now we can verify the homeomorphicity of $\theCurvedH$. In the proof given below, in (A*)--(C*), we handle certain limit points of $\domainTrapezoid$. Besides, there are five points which need a special treatment (cf.~\eqref{equation:vMonotoneGaussianComplexCST}) while using the Stieltjes inversion formula: $\pm \sqrt{2}$ and $0$, which are branch points of $\halfOfLogarithm$ (it takes the value $\complexInfty$ for them, the limits of $\theFunctionGmu$, however, are finite), and $\pm \sqrt{2+\gamma_{0}}$, for which $\theFunctionFmu$ is zero (they are also its branch points of order $2$). We require (A*)--(C*) also to deal with these points.
\begin{lemma}
\label{lemma:homeomorphicityOfH}
Consider the strip $\domainClosedLowerStrip = \{ w \in \CC : -\lfrac{\pi}{2} \leq \Im w \leq 0 \}$ and let $\domainRangeOfTZero = (-\infty, \lfrac{ \sqrt{3} \pi }{9} ]$.
The function $\theCurvedH \colon \domainTrapezoid \to \domainTrapezoidUnderH$ is a homeomorphism.
\end{lemma}

\begin{figure}[ht]
    \centering
    \includegraphics[scale=0.3]{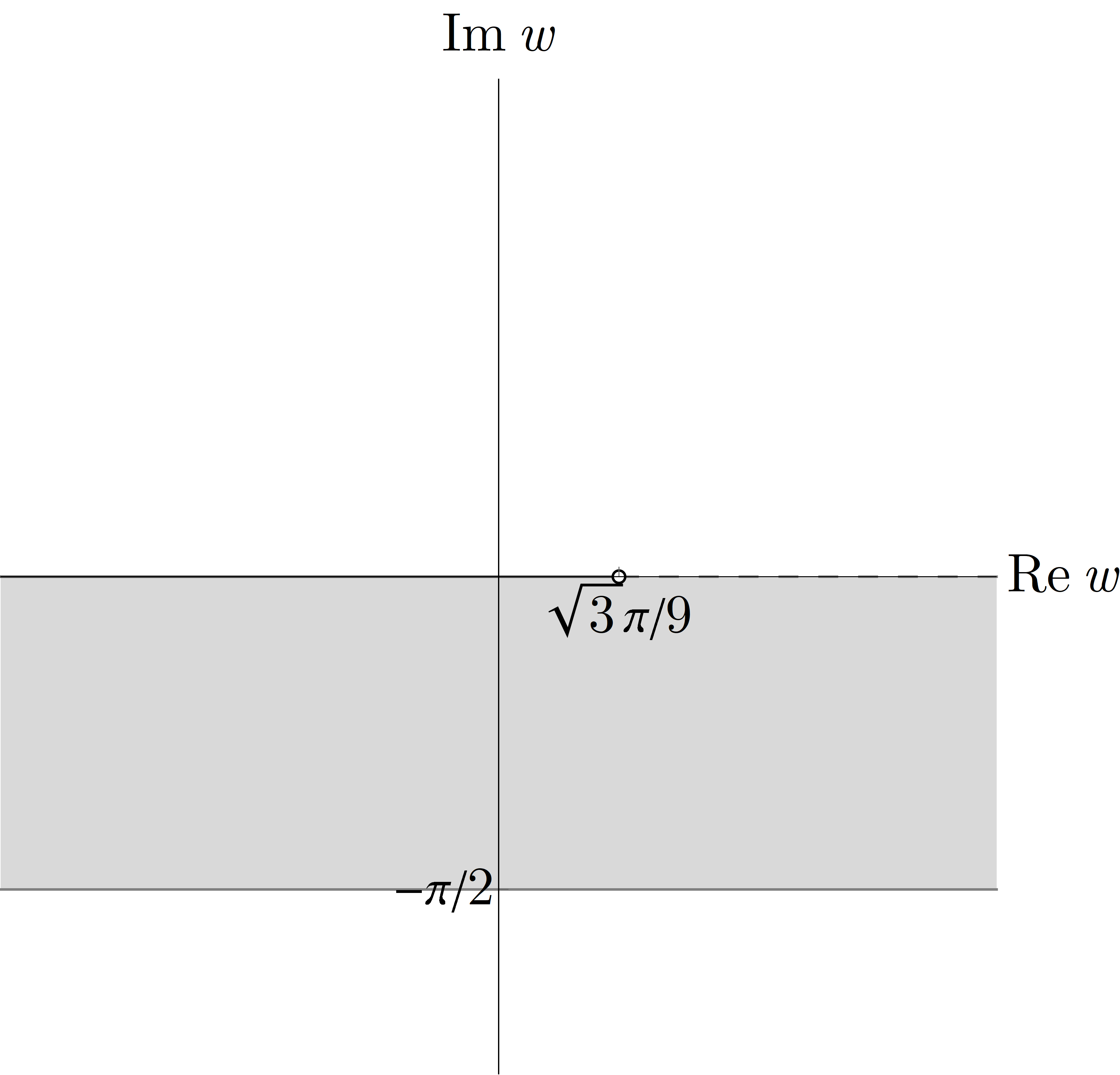}
    \caption{The set $\domainTrapezoidUnderH$}
    \label{figure:domainClosedLowerStrip}
\end{figure}

\begin{proof}
Let us first demonstrate that indeed $\theCurvedH(\domainTrapezoid) = \domainTrapezoidUnderH$. In order to do so, we will prove two statements. Let $(\eta_{n}, \xi_{n})_{n=1}^{\infty}$ be the sequence of elements from $\domainTrapezoid$ and let $R_{n}$, $u_{n}$ and $w_{n}$ be the same as in the proof of the previous lemma.
We have
\begin{itemize}
    \item[(A*)] $\xi_{n} \to -\infty
    \Leftrightarrow
    \Re w_{n} \to \infty$,
    \item[(B*)] $\lfrac{\xi_{n}}{\eta_{n}} \to (-1)^{+}
    \Leftrightarrow
    \Re w_{n} \to -\infty$.
\end{itemize}
In the proof we will use the statements~(A)--(C) from the proof of Lemma~\ref{lemma:propertiesOfTheUpperShell}. Of course,
\begin{equation}
\label{equation:limitPointsCLT1} 
\begin{split}
w_{n} =
& - \frac{\sqrt{3}}{6} \Bigg( \principalArgument(R_{n} e^{i \xi_{n}} - 1) - \principalArgument(u_{n} - \complexAdjoint{ u_{0} }) + 2 \bigg( n_{2} + \bigg \lceil \frac{\xi_{n} + \pi}{2 \pi} \bigg \rceil \bigg) \pi + \frac{13 \pi}{6} \Bigg) \\
& \ - \frac{\realLogarithm ( \lvert u_{n} - u_{0} \rvert \cdot \lvert u_{n} - \complexAdjoint{ u_{0} } \rvert )}{2} + \frac{i \eta_{n}}{2} \text{}
\end{split}
\end{equation}

\noindent \textit{Proof of~(A*).} The implication `$\Rightarrow$' follows directly from~(A). If $\Re w_{n} \to \infty$, then either $\lceil (\xi_{n} - \pi) /2 \pi \rceil \to -\infty$, or $\realLogarithm ( \lvert u_{n} - u_{0} \rvert \cdot \lvert u_{n} - \complexAdjoint{ u_{0} } \rvert ) \to \infty$, which in both cases means that $\xi_{n} \to -\infty$ --- in the latter one it follows from~(A), since $\Im u_{n} > 0$. The remaining terms in~\eqref{equation:limitPointsCLT1} are bounded and thus irrelevant.

\noindent \textit{Proof of~(B*).}
The statement is a consequence of~(B) and of the fact that the sequence $(\xi_{n})_{n=1}^{\infty}$ is bounded: from above, by the definition of $\domainTrapezoid$, and from below, by~(A*).

Fix $-\pi \leq \eta \leq 0$ and let us find the range of $\xi \mapsto \theCurvedH(\eta, \xi)$ for $\xi < -\eta$. In order to do so, due to the monotonicity of $\xi \mapsto \theCurvedHRe(\eta, \xi)$ demonstrated in the proof of Proposition~\ref{proposition:injectivityOfH}, it suffices to find two limits of $\theCurvedHRe(\eta, \xi)$: as $\xi$ approaches $-\infty$ and as $\xi$ approaches $-\eta$ from below. The former limit can be determined from~(A*):
\begin{equation*}
\lim_{\xi \to -\infty} \theCurvedHRe(\eta, \xi) = \infty \text{.}
\end{equation*}
The latter limit depends on $\eta$. Now, assume that $\eta_{n} = \eta$ for all $n$ and $\xi_{n} \to -\eta$. If $\eta < 0$, then $\lfrac{\xi_{n}}{\eta_{n}} \to (-1)^{+}$, hence, according to~(B*), we have $\Re w_{n} \to -\infty$. If $\eta = 0$, then after elementary computations, we get $\lfrac{\xi_{n}}{\eta_{n}} \to \infty$, which means, due to~(B), that $u_{n} \to 0$ (along $\Gamma_{0}$). Since $\theFunctionT$ is continuous,
\begin{equation*}
\theCurvedHRe(\eta, \xi_{n}) = \Re \theFunctionT(u_{n}) \to \Re \theFunctionT(0) = \theFunctionT(0) = \frac{\sqrt{3} \pi}{9} \text{,}
\end{equation*}
i.e. $\theCurvedH(\domainTrapezoid) = \domainTrapezoidUnderH$.

Of course, $\theCurvedH$ is continuous by the definition. It remains to show that so is $\inverseFunction{ \theCurvedH }$.
We drop the assumptions about the sequence $((\eta_{n}, \xi_{n}))_{n=1}^{\infty}$ from the previous paragraph --- we only require that its elements belong to $\domainTrapezoid$. In order to do so, let us first prove the following:
\begin{itemize}
    \item[(C*)] for any $s \geq 0$, $\big( (\eta_{n}, \xi_{n}) \to (0, 0) \wedge \lfrac{\xi_{n}}{\eta_{n}} \to \lfrac{1}{s} - 1 \big)
    \Leftrightarrow
    w_{n} \to \theRealValuedT(s)$
    with the convention $\lfrac{1}{0} = \infty$.
\end{itemize}
We only need to show the implication `$\Leftarrow$'
because the other one follows from~(C) and from the definition of $\theFunctionT$.
Let us then assume that $w_{n} \to \theRealValuedT(s)$, for some $s \geq 0$. Of course, by~\eqref{equation:limitPointsCLT1}, we have $\eta_{n} \to 0$. 
Let us show that also $\xi_{n} \to 0$. If some $\tilde{\xi} < 0$ was a limit point of $(\xi_{n})_{n=1}^{\infty}$, then, by the continuity of $\theCurvedH$, 
it would hold that $\theRealValuedT(s) = \theCurvedH(0, \tilde{\xi})$.
This is impossible, since by the definition of~$\theRealValuedT$ and by Lemma~\ref{lemma:homeomorphicityOfH}, we have $\theRealValuedT(s) \leq \lfrac{\sqrt{3}\pi}{9} < \theCurvedH(0, \tilde{\xi})$. By~(A*), it is also impossible that $-\infty$ is a limit point of $(\xi_{n})_{n=1}^{\infty}$. Thus $(\eta_{n}, \xi_{n}) \to (0, 0)$.
Now we prove that $\lfrac{\xi_{n}}{\eta_{n}} \to \lfrac{1}{s} - 1$. Let $\lfrac{1}{\tilde{s}} - 1$ be a limit point of $(\lfrac{\xi_{n}}{\eta_{n}})_{n=1}^{\infty}$ with $\tilde{s} \in [-\infty, \infty]$. By the proof of Lemma~\ref{lemma:propertiesOfTheUpperShell}, $\tilde{s} \geq 0$, whereas, by~(B*), $\tilde{s} \neq \infty$.
By the implication `$\Rightarrow$', the corresponding limit point of $(w_{n})_{n=1}^{\infty}$ is $\theRealValuedT(\tilde{s})$. The function $\theRealValuedT$ is strictly decreasing on $[0, \infty)$, therefore $\tilde{s} = s$, which proves~(C*).

It remains to prove that $\inverseFunction{\theCurvedH}$ is continuous (see Fig.~\ref{figure:domainTrapezoid} and Fig.~\ref{figure:domainClosedLowerStrip}).
Let $(\eta'_{n}, \xi'_{n})_{n=1}^{\infty}$ be a sequence of elements from $\domainTrapezoid$ such that $w'_{n} = \theCurvedH(\eta'_{n}, \xi'_{n}) \to \theCurvedH(\eta, \xi)$ for some $(\eta, \xi) \in \domainTrapezoid$.
By Lemma~\ref{lemma:exactFormOfH}, $\eta'_{n} \to \eta$, therefore it remains to show that $\xi'_{n} \to \xi$. Let $\tilde{\xi} \in [-\infty, -\eta]$ be a limit point of $(\xi'_{n})_{n=1}^{\infty}$. If it was equal to $-\infty$ or $-\eta$, then by~(A*)--(C*), the corresponding subsequence of $(\Re w'_{n})_{n = 1}^{\infty}$ would tend to $\infty$ or its corresponding limit points would belong to $[ -\infty, \sqrt{3} \pi /9 ] = \theRealValuedT \big( [0, \infty] \big)$, respectively, which is impossible. By the continuity of $\theCurvedH$, the corresponding limit point of $(w'_{n})_{n=1}^{\infty}$ is $\theCurvedH(\eta, \tilde{\xi})$ and, from the injectivity of $\theCurvedH$, we have $\tilde{\xi} = \xi$, which leads to the desired conclusion.
\end{proof}

Now, we are ready to show that $\theFunctionT$ is invertible and to write $\theInverseOfT$ in an implicit form. 
\begin{theorem}
\label{theorem:propertiesOfTheInvertedT}
The function $\theInverseOfT \colon \domainClosedLowerStrip \to \topologicalClosure{\domainOpenUpperShell} \setminus \{ u_{0} \}$ given by
\begin{equation}
\label{equation:theInverseOfT}
\theInverseOfT(w) =
\begin{cases}
\theInverseOfRealT(w) & \text{ if $w \in \domainRangeOfTZero$, } \\
\theCurvedf(\eta, \xi) & \text{ if $w = \theCurvedH(\eta, \xi) \in \domainTrapezoidUnderH$, for $(\eta, \xi) \in \domainTrapezoid$, }
\end{cases}
\end{equation}
is the inverse of $\theFunctionT$. Moreover, it is a continuous extension of $\theInverseOfRealT$ which is analytic on $\domainLowerStrip$. 
\end{theorem}

\begin{proof}
First, by~\eqref{eqation:theRangeOfTheCurvedf} and~\eqref{equation:topologicalClosureOfUpperShell}, we have
\begin{equation*}
\topologicalClosure{ \domainOpenUpperShell } \setminus \{ u_{0} \} = \theCurvedf( \domainTrapezoid ) \cup [0, \infty) \text{}
\end{equation*}
and Lemma~\ref{lemma:propertiesOfTheUpperShell} implies that $\theCurvedf( \domainTrapezoid ) \cap [0, \infty) = \emptyset$. The existence of $\theInverseOfT$ and its implicit form follow then from the definitions of $\theFunctionT$, $\theInverseOfRealT$ and $\theCurvedH$ and from the homeomorphicity of $\theCurvedf$ and $\theCurvedH$ (see Lemmas~\ref{lemma:changeOfCoordinates-Properties} and~\ref{lemma:homeomorphicityOfH}).

It suffices to show the continuity of $\theInverseOfT$. Let $(w_n)_{n=1}^{\infty}$ be a sequence such that $\domainClosedLowerStrip \ni w_{n} \to w \in \domainClosedLowerStrip$. Without loss of generality, it suffices to consider three cases.
First, let us assume that the elements of the sequence and its limit belong to $\domainTrapezoidUnderH$. In this case, the fact that $\theInverseOfT(w_n) \to \theInverseOfT(w)$ follows from the homeomorphicity of $\theCurvedf$ and $\theCurvedH$.
If the elements of the sequence and its limit belong to $\domainRangeOfTZero$, the same fact is implied by the definition of $\theInverseOfRealT$.
Consider the third case: let $(\eta_n, \xi_n)_{n=1}^{\infty}$ be a sequence of elements belonging to $\domainTrapezoid$ such that $w_n = \theCurvedH(\eta_n, \xi_n)$, for $n \in \NN_{+}$, and let $s \geq 0$ be such that $w = \theRealValuedT(s)$. The continuity of $\theInverseOfT$ follows from the fact that $u_{n} \coloneqq \theCurvedf(\eta_n, \xi_n) \to s$, implied by~(C) and~(C*) from the proofs of Lemmas~\ref{lemma:propertiesOfTheUpperShell} and~\ref{lemma:homeomorphicityOfH}, respectively.
\end{proof}

\section{Density}
\label{section:CLTdensity}
In this section, we give an exact form of the Cauchy--Stieltjes transform of the standard V-monotone Gaussian measure $\mu$. Then, using the Stieltjes inversion formula,
we demonstrate that $\mu$ is absolutely continuous with respect to the Lebesgue measure on $\RR$ and we find its density in an implicit form.

\begin{theorem}
\label{theorem:cauchyTransformOfVmonotoneStandardGaussian}
The reciprocal Cauchy--Stieltjes transform of $\mu$ is the extension of
\begin{equation}
\label{equation:reciprocalCauchyTransform}
\theFunctionFmu(z) = z \theInverseOfT \Bigg( \dfrac{1}{2} \principalLogarithm \bigg( 1 + \dfrac{2}{z^{2} - 2} \bigg) \Bigg) \text{,} \quad z \in \openedFirstQuadrant \cup i \cdot \positiveRealHalfAxis \cup (\sqrt{2+\gamma_{0}}, \infty) \text{}
\end{equation}
to $(\CC_{+} \cup \RR) \setminus [-\sqrt{2+\gamma_{0}}, \sqrt{2+\gamma_{0}}]$, given by $\theFunctionFmu(z) = -\complexAdjoint{\theFunctionFmu(-\complexAdjoint{z})}$ for $\Re z < 0$, and denoted by the same symbol.
\end{theorem}

\begin{proof}
By Proposition~\ref{proposition:domainCutStripANDhalfOfLogarithm} and Theorem~\ref{theorem:propertiesOfTheInvertedT}, the function given by~\eqref{equation:reciprocalCauchyTransform} is an extension of $x \mapsto \theFunctionFmu(x)$ defined in Theorem~\ref{proposition:vMonotoneGaussianRealCST} and is analytic on $\openedFirstQuadrant$. Therefore, $\theFunctionFmu$ is indeed the reciprocal Cauchy--Stieltjes transform of $\mu$ for $z \in \openedFirstQuadrant \cup i \cdot \positiveRealHalfAxis \cup (\sqrt{2+\gamma_{0}}, \infty)$. Since $\mu$ is symmetric with respect to $0$ (we recall that all its odd moments vanish), we must have $\theFunctionFmu(- \complexAdjoint{z}) = -\complexAdjoint{\theFunctionFmu(z)}$, due to~\eqref{equation:csTransformOfSymmetricMeasure}, and the theorem is proved.
\end{proof}

Below we demonstrate that $\mu$ is absolutely continuous with respect to Lebesgue measure and we find its density.
Let $\theExtendedFmu \colon \semiclosedFirstQuadrant \to \CC$ be given by $\theExtendedFmu(z) = z \theInverseOfT( \halfOfLogarithm(z) )$.
Define $\vCLT \colon \RR \to \RR$ as the even extension of
\begin{equation}
\label{equation:vCLTdensity}
\vCLT(x) = 
\begin{cases}
\dfrac{ \sqrt{2} }{ 2 \pi } \cdot e^{ \lfrac{\sqrt{3} \pi}{9} } & \text{for $x = 0$,} \\
\dfrac{\Im \theExtendedFmu(x)}{\pi \cdot \lvert \theExtendedFmu(x) \rvert^2} & \text{for $x \in (0, \sqrt{2+\gamma_{0}}) \setminus \{ \sqrt{2} \} $,} \\
\dfrac{\sqrt{6}}{4 \pi} & \text{for $x = \sqrt{2}$,} \\
0 & \text{for $x \geq \sqrt{2+\gamma_{0}}$.}
\end{cases}
\end{equation}
Let us now determine the standard V-monototone Gaussian measure.
\begin{theorem}
\label{theorem:vMonotoneStandardGaussianDistribution}
The measure $\mu$ is supported on $[-\sqrt{2+\gamma_{0}}, \sqrt{2+\gamma_{0}}]$ and is absolutely continuous with respect to the Lebesgue measure with the density function $\vCLT$.
\end{theorem}

\begin{proof}
The proof is based on Stieltjes inversion formula. Let us first prove that $\theFunctionGmu$ is continuously extendable from $\CC_{+}$ to $\CC_{+} \cup \big( \RR \setminus \{ \mp \sqrt{2+\gamma_{0}} \} \big)$. Let us show that the mentioned extension $\theExtendedGmu \colon \CC_{+} \cup \big( \RR \setminus \{ \mp \sqrt{2+\gamma_{0}} \} \big) \to \CC$ is defined by the extension of
\begin{equation}
\label{equation:theExtensionOfGmu}
\theExtendedGmu(z) =
\begin{cases}
\dfrac{1}{z \theInverseOfT(W(z))} & \text{for $z \in \semiclosedFirstQuadrant$,} \\
-\dfrac{\sqrt{2} i}{2} e^{\sqrt{3} \pi / 9} & \text{for $z = 0$,} \\
\dfrac{\sqrt{2} - \sqrt{6}i}{4} & \text{for $z = \sqrt{2}$,}
\end{cases}
\end{equation}
given by $\theExtendedGmu(z) = -\complexAdjoint{\theExtendedGmu(-\complexAdjoint{z})}$, for $\Re z < 0$.

Let $\halfOfLogarithmRestriction = \myRestriction{\halfOfLogarithm}{ (\topologicalClosure{ \openedFirstQuadrant } \setminus \{ 0, \sqrt{2} \}) }$. We can write
\begin{equation*}
\halfOfLogarithmRestriction(z) = \frac{1}{2} \principalLogarithmB \bigg( 1 + \frac{2}{z^2 - 2} \bigg) \text{,}
\end{equation*}
where $\principalLogarithmB$ is a modified logarithm function given by
\begin{equation*}
\principalLogarithmB(v) = \realLogarithm r + i \varphi \text{} \qquad \text{for $v = r \cdot e^{i \varphi}$ with $r > 0$ and $-\pi \leq \varphi < \pi$, }
\end{equation*}
which differs from the one given by~\eqref{equation:theSuitableLogarithm} only for $v \in (-\infty, 0)$. \triExclamation
By Proposition~\ref{proposition:domainCutStripANDhalfOfLogarithm}, we have
\begin{equation*}
\label{equation:rangeOfHalfLog}
\begin{split}
\halfOfLogarithmRestriction(\openedFirstQuadrant) & = \topologicalInterior{ \domainClosedLowerStrip } \text{,} \\
\halfOfLogarithmRestriction( (\sqrt{2}, \infty) \cup i \cdot \positiveRealHalfAxis ) & = \frac{1}{2} \principalLogarithm( \positiveRealHalfAxis \setminus \{ 1 \} ) = \RR \setminus \{ 0 \} \text{,} \\
\halfOfLogarithmRestriction( (0, \sqrt{2}) ) & = \frac{1}{2} \realLogarithm \lvert (-\infty, 0) \rvert - \frac{i \pi}{2} = \RR - \frac{i \pi}{2} \text{.}
\end{split}
\end{equation*}
Let us now show that the function $\halfOfLogarithmRestriction$ is invertible and
\begin{equation*}
\inverseFunction{\halfOfLogarithmRestriction}(w) = \sqrt{2 + \frac{2}{e^{2w}-1}} \text{,}
\end{equation*}
for $w \in \domainClosedLowerStrip \setminus \{ 0 \}$, where by the square-root function, we mean a branch
\begin{equation*}
\sqrt{v} = \sqrt{r} \cdot e^{i \varphi / 2} \text{,} \qquad \text{for $v = r \cdot e^{i \varphi}$ with $r > 0$ and $-\lfrac{\pi}{2} < \varphi < \lfrac{3 \pi}{2}$, }
\end{equation*}
whose restriction $\sqrt{\text{ . }} \colon \closedUpperHalfPlane \to \topologicalClosure{ \openedFirstQuadrant }$ is a homeomorphism. The procedure of inverting looks as follows: let $z \in \topologicalClosure{ \openedFirstQuadrant } \setminus \{ 0, \sqrt{2} \}$ and let $w = \halfOfLogarithmRestriction(z)$. By~\eqref{equation:rangeOfHalfLog}, we have $w \in \domainClosedLowerStrip \setminus \{ 0 \}$ and
\begin{equation}
\label{equation:invertingOfHalfLog}
\begin{aligned}
1 + \frac{2}{z^2 - 2} & = e^{2w} && \in \topologicalClosure{ \CC_{-} } \setminus \{ 0, 1 \} \text{,} \\
z^2 & = 2 + \frac{2}{e^{2w} - 1} && \in \topologicalClosure{ \CC_{+} } \setminus \{ 0, 2 \} \text{,} \\
z & = \sqrt{2 + \frac{2}{e^{2w} - 1}} && \in \topologicalClosure{ \openedFirstQuadrant } \setminus \{ 0, \sqrt{2} \} \text{.}
\end{aligned}
\end{equation}
The choice of the branch of the square-root function is based on this procedure. Looking at the just determined images and the inverse of $\halfOfLogarithmRestriction$ given by $z$, we conclude that $\halfOfLogarithmRestriction \colon (\topologicalClosure{ \openedFirstQuadrant } \setminus \{ 0, \sqrt{2} \}) \to \domainClosedLowerStrip \setminus \{ 0 \}$ is a homeomorphism.
Let us now prove that
\begin{equation}
\label{equation:halfOfLogarithmInverse}
\inverseFunction{\halfOfLogarithmRestriction}(w) = e^{w} \cdot \sqrt{\frac{2}{e^{2w} - 1}} \text{.}
\end{equation}
Indeed, if $w \in \domainClosedLowerStrip \setminus \{ 0 \}$, then $\lfrac{2}{(e^{2w} - 1)} \in \closedUpperHalfPlane$, therefore the $\rhs$ of the above equality (which we denote by $z$) is well defined and continuous; so is the $\lhs$. Moreover, since both $\inverseFunction{\halfOfLogarithmRestriction}(w)$ and $z$ are always nonzero, it must hold that either $\inverseFunction{\halfOfLogarithmRestriction}(w) = z$ for each suitable $w$, or $\inverseFunction{\halfOfLogarithmRestriction}(w) = -z$ for each suitable $w$. But the former equality is obvious for $w \in (0, \infty)$ and thus~\eqref{equation:halfOfLogarithmInverse} holds for all $w \in \domainClosedLowerStrip \setminus \{ 0 \}$.

Since our measure is symmetric with respect to $0$, it suffices to consider only $x \geq 0$ and $\theExtendedFmu$ and $\theExtendedGmu$ for arguments in $\topologicalClosure{ \openedFirstQuadrant }$.
Let $x \geq 0$ and let $(z_{n})_{n=1}^{\infty}$ be a sequence convergent to $x$, with the elements belonging to $\semiclosedFirstQuadrant$. Let $w_{n} = \halfOfLogarithm(z_{n})$. Now, without loss of generality, we assume that there are two possibilities (see Fig.~\ref{figure:domainClosedLowerStrip}): either there exists a sequence $(\eta_{n}, \xi_{n})_{n=1}^{\infty}$ of elements from $\domainTrapezoid$ such that $w_{n} = \theCurvedH(\eta_{n}, \xi_{n})$ in which case, for any positive $n$, we put $u_{n} = \theCurvedf(\eta_{n}, \xi_{n})$, or (only if $x \leq \sqrt{2 + \gamma_{0}}$) there exists a sequence $(u_{n})_{n=1}^{\infty}$ of nonnegative numbers such that $w_{n} = \theRealValuedT(u_{n})$. In the latter case, $u_{n} \in \RR$, whereas in the former one, $\Im u_{n} \neq 0$, for any positive $n$.
We will use~(A)--(C), the equality~\eqref{equation:theLimitOfEsEnPrim} and~(A*)--(C*) from the proofs of the mentioned lemmas and of Lemma~\ref{lemma:homeomorphicityOfH} while computing certain limits below. 
Let us consider four cases.

\noindent \textbf{Case~1: $x = 0$.} In this case, $\Re w_{n} \to -\infty$. Let us demonstrate that
\begin{equation}
\label{equation:thmVMonotoneStandardGaussianDistribution1}
z_{n} \cdot \theInverseOfT(w_{n}) = \sqrt{\frac{2}{e^{2w_{n}} - 1}} \cdot e^{w_{n}} \cdot u_{n} \to \sqrt{2} i \cdot e^{-\sqrt{3}\pi / 9} \text{.}
\end{equation}
If $u_{n} \in \RR$ for each $n$, then it follows immediately from~\eqref{equation:theRealVersionOfT}.
Otherwise, by~(B*) from the proof of Lemma~\ref{lemma:homeomorphicityOfH}, we have $\lfrac{\xi_{n}}{\eta_{n}} \to (-1)^{+}$. We now show that this implies that
\begin{equation}
\label{equation:thmVMonotoneStandardGaussianDistribution2}
u_{n} e^{w_{n}} \to e^{-\sqrt{3} \pi / 9} \text{.}
\end{equation}
Indeed, let us write
\begin{equation*}
u_{n} e^{w_{n}} = \lvert u_{n} \rvert \cdot e^{w_{n} - i \lfrac{\eta_{n}}{2}} \cdot e^{i (\principalArgument u_{n} + \lfrac{\eta_{n}}{2})} \text{}
\end{equation*}
and show that
\begin{equation}
\label{equation:eq1InItemE}
\lvert u_{n} \rvert \cdot e^{w_{n} - i \lfrac{\eta_{n}}{2}} \to e^{-\sqrt{3} \pi / 9} \text{.}
\end{equation}
By~\eqref{equation:theExactFormOfTheCurvedH}, we have
\begin{equation*}
\begin{split}
\lvert u_{n} \rvert \cdot e^{w_{n} - i \lfrac{\eta_{n}}{2}}
& = \exp \Bigg ( - \frac{\sqrt{3}}{6} \bigg( \principalArgument(R_{n} e^{i \xi_{n}} - 1) - \principalArgument( u_{n} - \complexAdjoint{ u_{0} } ) + 2 \cdot \bigg \lceil \frac{\xi_{n} - \pi}{2 \pi} \bigg \rceil \cdot \pi + \frac{\pi}{6} \bigg) \\
& \ \cdot \sqrt{ \frac{\lvert u_{n} \rvert^2}{\lvert u_{n} - u_{0} \rvert \cdot \lvert u_{n} - \complexAdjoint{ u_{0} } \rvert} } \Bigg ) \text{,}
\end{split}
\end{equation*}
where $R_{n} = e^{-\sqrt{3}(\eta_{n} + \xi_{n})}$. For $n$ large enough, we have $2 \Re u_{n} > 1$. Indeed, since $R_{n} > 0$ and $\lfrac{\eta_{n}}{\xi_{n}} \to (-1)^{+}$, we must have $\xi_{n} > 0$. Since $\eta_{n} + \xi_{n} < 0$ and $-\pi \leq \eta_{n} \leq 0$, we have $\xi_{n} < \pi$. Therefore $\sin \xi_{n} > 0$, i.e. $\Re u_{n} > 1 / 2$. By the definition of $\theCurvedf$ (see Definition~\ref{definition:theCurvedf}), we have
\begin{equation*}
R_{n} e^{i \xi_{n}} - 1 = i \cdot (u_{n} - u_{0}) \cdot \frac{R_{n}^2 - 1}{\sqrt{3}} \text{,}
\end{equation*}
therefore, since $\Re (u_{n} - u_{0}) > 0$, we have
\begin{equation*}
\principalArgument (R_{n} e^{i \xi_{n}} - 1) = \principalArgument (i \cdot (u_{n} - u_{0}) ) = \frac{\pi}{2} + \principalArgument(u_{n} - u_{0}) \text{.}
\end{equation*}
Moreover, since $\Re (u_{n} - \complexAdjoint{ u_{0} }) > 0$, we use Lagrange Mean Value Theorem, getting
\begin{equation*}
\principalArgument(u_{n} - u_{0}) - \principalArgument (u_{n} - \complexAdjoint{ u_{0} }) = \frac{\sqrt{3} \Re (u_{n} - u_{0})}{\lvert u'_{n} \rvert^2} \text{,}
\end{equation*}
where $\Re u'_{n} = \Re (u_{n} - u_{0})$ and $\Im (u_{n} - u_{0}) < \Im u'_{n} < \Im (u_{n} - \complexAdjoint{ u_{0} })$. Since $\lvert u_{n} \rvert \to \infty$ (by~(A) from the proof of Lemma~\ref{lemma:propertiesOfTheUpperShell}), we get
\begin{equation*}
\principalArgument (R_{n} e^{i \xi_{n}} - 1) - \principalArgument (u_{n} - \complexAdjoint{ u_{0} }) \to \frac{\pi}{2} \text{.}
\end{equation*}
Since $\xi_{n} \in (0, \pi)$, we have $\lceil (\xi_{n} - \pi) / 2 \pi \rceil = 0$ and thus~\eqref{equation:eq1InItemE} holds.

It suffices to show that $\principalArgument u_{n} + \lfrac{\eta_{n}}{2} \to 0$. Without loss of generality, we assume that $\eta_{n} \to \eta$. From Lemma~\ref{lemma:changeOfCoordinates-Properties} we get $\Re u_{n} > 0$ and $\Im u_{n} > 0$, hence $\principalArgument u_{n} \in ( 0, \lfrac{\pi}{2} )$. We consider two cases. First, let $-\pi < \eta \leq 0$. From~\eqref{equation:behaviourOfRealAndImaginaryPartsOfw} we see that $\Re u_{n} \to \infty$ and
\begin{equation*}
\frac{\Im u_{n}}{\Re u_{n}} = \frac{1 - \cos \xi_{n}}{\sin \xi_{n}} \cdot \big( 1 + o(1) \big) + o(1) \to \tan \bigg( -\frac{\eta}{2} \bigg) \text{.}
\end{equation*}
If $\eta = -\pi$, then, also from~\eqref{equation:behaviourOfRealAndImaginaryPartsOfw}, $\Im u_{n} \to \infty$ and
\begin{equation*}
\frac{\Re u_{n}}{\Im u_{n}} = \frac{\sin \xi_{n}}{1 - \cos \xi_{n}} \cdot \big( 1 + o(1) \big) + o(1) \to 0^{+} \text{.}
\end{equation*}
This yields
\begin{equation*}
\frac{\Im u_{n}}{\Re u_{n}} \to \tan \bigg( -\frac{\eta}{2} \bigg) \text{}
\end{equation*}
with the convention $\tan (\lfrac{\pi}{2}) = \infty$. Therefore $\principalArgument u_{n} + \eta_{n} \to 0$ and the proof of~\eqref{equation:thmVMonotoneStandardGaussianDistribution2} and thus that of~\eqref{equation:thmVMonotoneStandardGaussianDistribution1} is complete. Note that the limit in~\eqref{equation:thmVMonotoneStandardGaussianDistribution1} is pure imaginary, which implies that
\begin{equation*}
\lim \limits_{z \to 0 + 0^{+} \cdot i} \theFunctionGmu(z) = -\frac{\sqrt{2} i}{2} e^{\sqrt{3} \pi / 9} \text{,}
\end{equation*}
accordingly to the fact that $x \mapsto \Re \theFunctionGmu(x+iy)$ is an odd, and $x \mapsto \Im \theFunctionGmu(x+iy)$ is an even function.

\noindent \textbf{Case~2: $x \in (0, \sqrt{2})$.} From Theorem~\ref{theorem:propertiesOfTheInvertedT} we have
\begin{equation*}
\theFunctionFmu(z_{n}) \to x \theInverseOfT \Bigg( \frac{1}{2} \realLogarithm \bigg( \dfrac{2}{2 - x^{2}} - 1 \bigg) - \dfrac{\pi i}{2} \Bigg) = \theFunctionFmu(x) \text{.}
\end{equation*}

\noindent \textbf{Case~3: $x = \sqrt{2}$.} We see that $\Re w_{n} \to \infty$ and by the definition of $\theInverseOfT$ and, by~(A*) from the proof of Lemma~\ref{lemma:homeomorphicityOfH}, we see that $\theFunctionFmu(w_{n}) \to \lfrac{(\sqrt{2} + \sqrt{6}i)}{2}$. 

\noindent \textbf{Case~4: $x > \sqrt{2}$.} Again, Theorem~\ref{theorem:propertiesOfTheInvertedT} implies that
\begin{equation*}
\theFunctionFmu(z_{n}) \to x \theInverseOfT \Bigg( \frac{1}{2} \realLogarithm \bigg( 1 + \dfrac{2}{x^{2} - 2} \bigg) \Bigg) = \theFunctionFmu(x) \text{.}
\end{equation*}

What we just have proved is that the function $\theExtendedGmu$ is continuous. But its limit as $z$ approaches $\mp \sqrt{2+\gamma_{0}}$ is $\complexInfty$ since $\theFunctionFmu(x) = 0$ for $x = \sqrt{2 + \gamma_{0}}$. We now show that the function
\begin{equation*}
z \mapsto \sqrt{ \lvert z^2 - (2+\gamma_0) \rvert } \cdot \lvert \theExtendedGmu(z) \rvert
\end{equation*}
is bounded on any closed upper half-disk of center $0$ (with the points $\mp \sqrt{2 + \gamma_{0}}$ excluded). Note that $w_{n} \to \halfOfLogarithm(\sqrt{2+\gamma_{0}}) = \lfrac{\sqrt{3}\pi}{9}$. By~\eqref{equation:theRealVersionOfTInTheIntegralForm} and Theorem~\ref{theorem:propertiesOfTheInvertedT}, there exists a sequence $(u_{n})_{n=1}^{\infty}$ such that 
\begin{equation*}
w_{n} = \int_{u_{n}}^{1} \frac{\omega}{\omega^2 - \omega + 1} \, \diff{\omega} \text{,}
\end{equation*}
for each positive $n$. Since $u_{n} \to 0$, this integral is well-defined (i.e. there is no problem with the poles of the integrand). Using the definition of $\theInverseOfT$ and the Maclaurin expansion of the above integral, we obtain
\begin{equation*}
\frac{\lvert \theInverseOfT(w_{n}) \rvert}{\sqrt{\lvert z_{n} - \sqrt{2 + \gamma_{0}} \rvert}} = \frac{\lvert \theInverseOfT(w_{n}) \rvert}{\sqrt{\big \lvert w_{n} - \frac{ \sqrt{3} \pi}{9} \big \rvert}} \cdot \sqrt{ \Bigg \lvert \frac{w_{n} - \frac{ \sqrt{3} \pi}{9}}{z_{n} - \sqrt{2 + \gamma_{0}}} \Bigg \rvert} = \frac{\lvert u_{n} \rvert}{\sqrt{\lvert -u_{n}^2 + o(u_{n}^2) \rvert}} \cdot \sqrt{\lvert 1 + o(1) \rvert} \to 1 \text{.}
\end{equation*}

We are now in the position to show that $\mu$ is absolutely continuous with respect to Lebesgue measure on $\RR$ and its density function is
\begin{equation*}
-\frac{1}{\pi} \Im \theExtendedGmu(x) \text{,}
\end{equation*}
for $x \in \RR \setminus \{ -\sqrt{2+\gamma_{0}}, \sqrt{2+\gamma_{0}} \}$. Indeed, if $\phi$ is continuous and compactly supported on $\RR$, we have
\begin{equation*}
\label{equation:dominatedConvergenceTheorem}
\lim_{y \to 0^{+}} \int_{-\infty}^{\infty} \phi(x) \Im \theFunctionGmu(x+yi) \diff{x} = \int_{-\infty}^{\infty} \phi(x) \Im \theExtendedGmu(x) \diff{x} \text{,}
\end{equation*}
which follows from Lebesgue's dominated convergence theorem with the majorant 
\begin{equation*}
\frac{\lvert \phi(x) \rvert}{\pi} \cdot \frac{\sup_{z} \big( \sqrt{ \lvert z^2 - (2+\gamma_0) \rvert } \lvert \theFunctionGmu(z) \rvert \big)}{\sqrt{ \lvert x + \sqrt{2+\gamma_0} \rvert} \sqrt{\lvert x - \sqrt{2+\gamma_0} \rvert }} \text{,}  
\end{equation*}
where the supremum is taken over $\supp{\phi} + i \cdot (0, \varepsilon]$, for some $\varepsilon > 0$, and where $\lvert x \rvert \neq \sqrt{2+\gamma_0}$. This proves that $\vCLT$ is indeed the density of $\mu$.
\end{proof}

\begin{figure}[ht]
    \centering
    \includegraphics[scale=0.2]{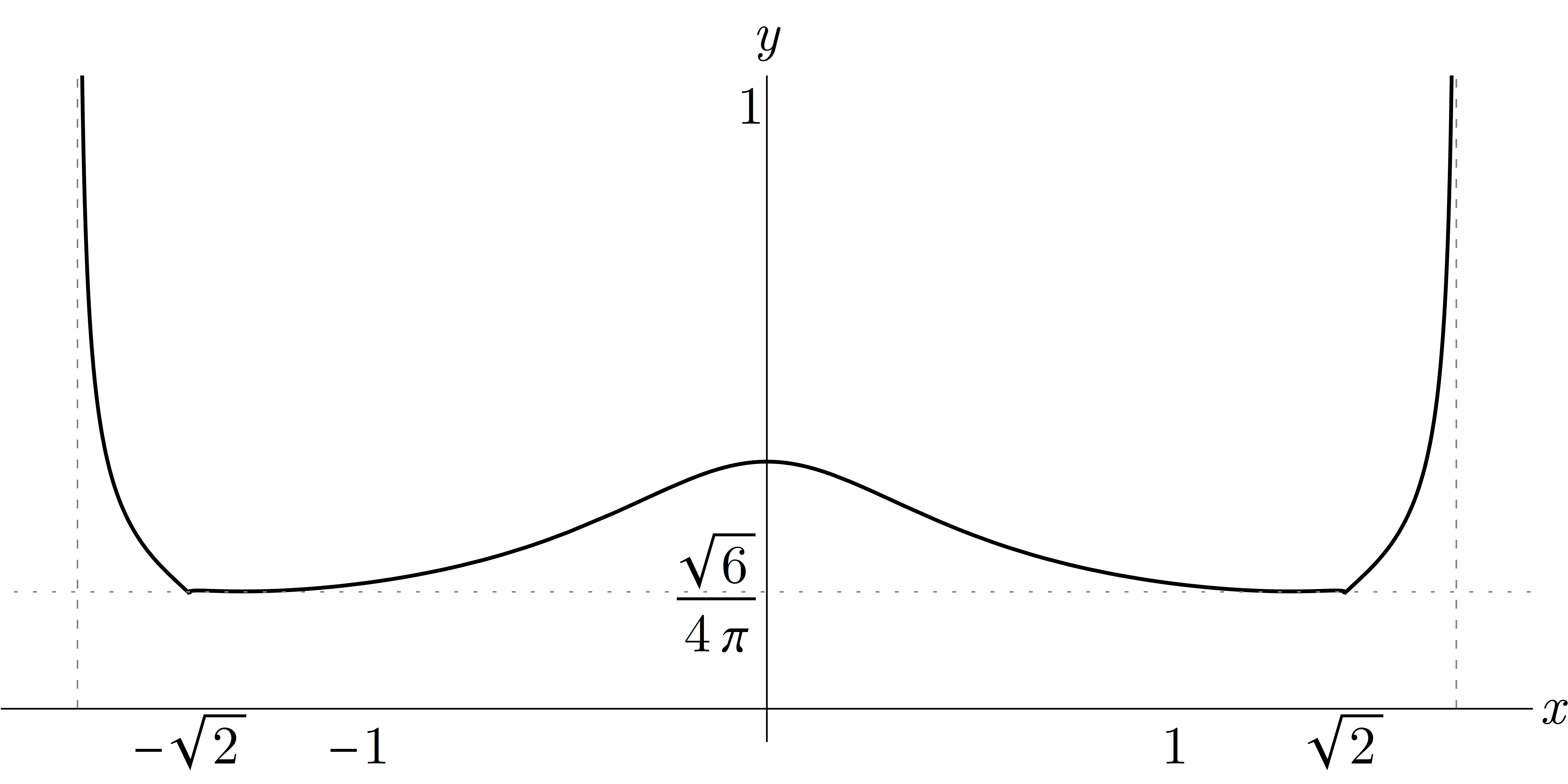}
    \caption{The graph of the density of the standard V-monotone Gaussian measure, $x \in (-\sqrt{2+\gamma_{0}}, \sqrt{2+\gamma_{0}})$}
    \label{figure:vCLT_density}
\end{figure}

The graph of the density is presented in Fig.~\ref{figure:vCLT_density}. It was generated by means of~\eqref{equation:vCLTdensity}, where the values of $\theInverseOfT$ was computed parametrically using~\eqref{equation:theFunctionH}: for $x \in (0, \sqrt{2})$ we put $\eta = -\pi$ and $\xi < \pi$, whereas for $x \in (\sqrt{2}, \sqrt{2+\gamma_{0}})$ we put $\eta = 0$ and $\xi < 0$.

The graph is symmetric with respect to the line $x = 0$ and is a $\mathcal{C}^1$-class function on $(-\sqrt{2+\gamma_{0}}, \sqrt{2+\gamma_{0}}) \setminus \{ \pm \sqrt{2} \}$, but is continuous on the whole open interval. The function has a local maximum at $0$ given by $\vCLT(0) = \lfrac{\sqrt{2}}{2 \pi} \cdot e^{\sqrt{3} \pi / 9}$, is strictly increasing on $(0, x_{1})$ and strictly decreasing on $(x_{2}, \sqrt{2+\gamma_{0}})$ for some $x_{1} < \sqrt{2} < x_{2}$. In the neighborhood of $\sqrt{2}$, the function oscillates and its amplitude and wavelength tend to $0$ exponentially as the argument tends to $\sqrt{2}$.
At the ends of the support the graph has one-sided vertical asymptotes --- the function behaves asymptotically as $\lfrac{1}{\sqrt{(2 + \gamma_{0}) - x^2}}$ as $x \to (\mp \sqrt{2+\gamma_{0}})^{\pm}$.

\section*{Acknowledgements}
I would like to thank Professor Romuald Lenczewski for many remarks and continuous support. I would also like to thank Professor Janusz Wysocza\'{n}ski for valuable comments. I also thank my colleagues Dariusz Kosz, Pawe\l{} Plewa and Artur Rutkowski for several helpful comments.

\normalsize

\bibliographystyle{plain}
\bibliography{main.bib}

\end{document}